\documentclass[11pt,draft]{article}

\usepackage[top=3cm, bottom=3cm, left=2.5cm, right=2.5cm]{geometry}

\usepackage[T1]{fontenc}
\usepackage{latexsym,amssymb,enumerate}
\usepackage{amsmath}
\usepackage{amsthm}
\usepackage[mathscr]{eucal}
\usepackage{color}
\usepackage{xcolor}
\usepackage[abbrev]{amsrefs}
\usepackage{array}
\usepackage{arydshln}
\usepackage{here}
\usepackage{graphicx}
\usepackage{multirow}

\newtheorem{Theorem}{\bf Theorem}[section]
\newtheorem{Lemma}{\bf Lemma}[section]
\newtheorem{Proposition}{\bf Proposition}[section]
\newtheorem{Corollary}{\bf Corollary}[section]
\newtheorem{Remark}{\bf Remark}[section]
\newtheorem{Example}{\bf Example}[section]
\newtheorem{Definition}{\bf Definition}[section]

\newenvironment{theorem}{\begin{Theorem}}{\end{Theorem}}
\newenvironment{lemma}{\begin{Lemma}}{\end{Lemma}}

\newenvironment{corollary}{\begin{Corollary}}{\end{Corollary}}
\newenvironment{remark}{\begin{Remark}}{\end{Remark}}
\newenvironment{example}{\begin{Example}}{\end{Example}}

\makeatletter
\def\senbun#1(#2)#3({\@senbun(#2)(}
\def\@senbun(#1,#2)(#3,#4){%
   \@tempdima#1\p@ \advance\@tempdima#3\p@
   \divide\@tempdima\tw@
   \@tempdimb#2\p@ \advance\@tempdimb#4\p@
   \divide\@tempdimb\tw@
   \edef\@senbun@temp{\noexpand\qbezier(#1,#2)%
      (\strip@pt\@tempdima,\strip@pt\@tempdimb)(#3,#4)}%
   \@senbun@temp}
\makeatother

\allowdisplaybreaks[4]

\numberwithin{equation}{section}

\def\XXint#1#2#3{{\setbox0=\hbox{$#1{#2#3}{\int}$}
\vcenter{\hbox{$#2#3$}}\kern-.5\wd0}}

\title{Multiplicity of singular solutions for semilinear elliptic equations with superlinear source terms}
\author{Yohei Fujishima and Norisuke Ioku}
\author{
        Yohei Fujishima\footnote{e-mail address: fujishima@shizuoka.ac.jp} \\ \\
        {\small Department of Mathematical and Systems Engineering} \\
        {\small Faculty of Engineering, Shizuoka University} \\
        {\small 3-5-1 Johoku, Hamamatsu 432-8561, Japan} \\ \\
        Norisuke Ioku\footnote{e-mail address: ioku@tohoku.ac.jp} \\ \\
        {\small Mathematical Institute, Tohoku University} \\
        {\small Aramaki 6-3, Sendai 980-8578, Japan}
        }
\date{}
\begin{document}
\maketitle

\abstract{
This paper investigates the multiplicity of singular solutions for the nonlinear elliptic equation
$-\Delta u =f(u)$ near the origin.
Applying the classification of nonlinear functions and the transformation, which were developed in \cite{FI}, we generalize the multiplicity results known for the concrete 
model nonlinearity $f(u)=u^p$ with $\frac{N}{N-2}<p<\frac{N+2}{N-2}$.
Our result applies to various nonlinearities, such as 
$f(s)=s^p+s^r$ 
with $0<r<p$,
$f(s)=s^p(\log s)^r$ with $r\in \mathbb{R}$,
$f(s)=s^p\exp((\log s)^r)$ with $0<r<1$ 
and
$f(s)=s^p+s^r(\log s)^{\beta}$ with $0<r<p$ and $\beta \in \mathbb{R}$,
for $\frac{N}{N-2}<p<\frac{N+2}{N-2}$. 
}

\ \\
\noindent
{\small
\vspace{5pt}
\vspace{5pt}
}

\section{Introduction}

We consider existence of a singular solution for a nonlinear elliptic equation. 
We say that a positive function $u>0$ 
is a singular solution of 
\begin{equation} 
  \label{eq:1.1}
  -\Delta u = f(u) 
\end{equation}
in a neighborhood of the origin $0\in \mathbb{R}^N$ (which will be simply called a singular solution of \eqref{eq:1.1}) if 
$u\in C^2(B_{r_0}\setminus \{0\})$ satisfies \eqref{eq:1.1} in $B_{r_0}\setminus \{0\}$
and $u(x) \to 0$ as $x\to 0$ for some $r_0>0$. 
Here $N\ge 1$ is the dimension of an underlying space, 
$B_{r_0}$ denotes the ball of radius $r_0>0$ centered at $0$ 
and a positive function $f\in C^2((0,\infty))$ denotes the nonlinear term of the problem. 
The aim of this paper is to prove the multiplicity of singular solutions of \eqref{eq:1.1} 
under certain assumptions for a nonlinear term $f$. 
In particular, the multiplicity of singular solutions will be shown when $f$ is Sobolev subcritical, 
but our method can be applied to a variety of nonlinear problems other than purely power type nonlinear problem.

Let us recall some multiplicity and uniqueness results of positive singular solutions 
of 
\eqref{eq:1.1} for the particular case $f(u)=u^p$ for $N\ge 3$, i.e.,
\begin{equation}\label{eq:1.2}
-\Delta u =u^p.
\end{equation}
The critical exponents 
\begin{equation}
  \label{eq:1.3}
  p_\mathrm{c} := \frac{N}{N-2}, 
  \qquad
  p_\mathrm{S} := \frac{N+2}{N-2}, 
\end{equation}
play a key role in the study of the multiplicity of solutions for \eqref{eq:1.2}.
When $1<p<p_\mathrm{c}$, 
Lions~\cite{L}
showed that 
\eqref{eq:1.2} has infinitely many singular solutions which diverges at the same rate
as the fundamental solution $|\cdot|^{2-N}$
of the Laplace equation (see also \cite{BL}).
If $p>p_\mathrm{c}$, then 
\eqref{eq:1.2} posed in a punctured domain $\mathbb{R}^N \setminus \{0\}$ 
admits an explicit singular solution $v_{p}$ defined by 
\begin{equation}
  \label{eq:1.4}
  v_{p}(x) := L_p|x|^{-\frac{2}{p-1}}, 
  \quad 
  L_p := \left(
    \frac{2}{p-1}\left(
      N-2-\frac{2}{p-1}
    \right)
  \right)^\frac{1}{p-1}, 
\end{equation}
  which is also a distributional solution of \eqref{eq:1.2} in the whole space $\mathbb{R}^N$. 

When $p_\mathrm{c} \le p<p_\mathrm{S}$, Ni-Sacks~\cite{NS}
constructed infinitely many singular solutions of \eqref{eq:1.2}.
On the other hand, when $p>p_\mathrm{S}$, 
Serrin-Zou~\cite{SZ} proved that the function~\eqref{eq:1.4} is the unique radial singular solution of \eqref{eq:1.2}.
As for the behavior of the singular solutions,
Chen-Lin~\cite{CL} showed the precise asymptotic expansion of the singular solutions
for $p_\mathrm{c}<p<p_\mathrm{*}$ (see also \cite{GNW} for related topics), 
where $p_*\in (p_\mathrm{c},p_\mathrm{S})$ is an exponent which will be used later (see \eqref{eq:1.14}).
For the lower border line case $p=p_\mathrm{c}$,
the behavior of singular solutions, $u(x) \simeq |x|^{2-N}(-\log|x|)^{\frac{2-N}2}$ near $0$, are revealed by
Aviles~\cite{A}.
For the upper border line case $p=p_\mathrm{S}$,
Caffarelli-Gidas-Spruck~\cite{CGS} proved 
that \eqref{eq:1.2} has a continuum of singular solutions. 
It should be mentioned that
these singular solutions for $p_\mathrm{c}\le p$ also satisfy equation~\eqref{eq:1.2} in a non-punctured ball in the sense of distributions, while for $1<p<p_\mathrm{c}$ the singular solutions satisfies 
$-\Delta u=u^p+\alpha \delta_0$ in the sense of distributions 
for some $\alpha>0$,
where $\delta_0$ is the delta function supported at the origin.

As a corresponding case to $p=\infty$
for $N\ge 3$,
exponential type nonlinearities 
were investigated and singular solutions to 
\eqref{eq:1.1} were constructed by Mignot-Puel \cite{MP} for $f(u)=e^u$,
Kikuchi-Wei~\cite{KW} for $f(u)=e^{u^q} (q>0)$,
Ghergu-Goubet~\cite{GG} for multiple exponential nonlinearities.
More general nonlinearities
in the form of $f(u)=u^p+\sigma(u)$
or $f(u)=e^u+\sigma(u)$ with a lower order term $\sigma(u)$
have been widely studied. 
Since there are a vast number of researches, we only refer to
\cites{J,M2,M3,MN2,MN3}; see also references therein.
Meanwhile, some recent developments for 2-dimensional problems can be found in 
\cites{FIRT,IKNW,Kumagai,Naimen}.

Recently, semilinear elliptic and parabolic problems with general superlinear source term $f(u)$, which is not necessarily a perturbed form from the model case $f(u)=u^p$,
are considered. 
One of triggers to accelerate generalization is 
the quasi scale invariance\footnote{
  Although the original discussion in \cite{F} was done for 
  parabolic problems, 
  it can be applied to elliptic equations by excluding the time variable.
} 
introduced in the first named author's work \cite{F}, that is,
\begin{equation}\label{eq:1.5}
u_\lambda(x):=F^{-1}\left(
\lambda^{-2} F(u(\lambda x))
\right) \quad (\lambda>0),
\quad \mbox{where} \,\,\, F(s):=\int_s^\infty \frac{d\eta}{f(\eta)}, 
\end{equation}
and
$F^{-1}$ is the inverse function of $F$ which is a non-increasing function. 
Then, for a solution $u$ of \eqref{eq:1.1},
the
scaled function $u_{\lambda}$ 
satisfies
\begin{equation}\label{eq:1.6}
-\Delta u_\lambda=f(u_\lambda)+ \frac{|\nabla u_\lambda|^2}{f(u_\lambda)F(u_\lambda)}\Big[
f'(u) F(u)-f'(u_\lambda)F(u_\lambda)
\Big].
\end{equation}
Inspired by \eqref{eq:1.5} and \eqref{eq:1.6}, 
the authors introduced in \cite{FI}
a classification of nonlinearities by the limit
\begin{equation}
  \label{eq:1.7}
  q_f = \lim_{u\to\infty} f'(u)F(u), 
\end{equation}
and a transformation as follows:
if $v$
satisfies 
\(
-\Delta v=g(v),
\) 
then
\begin{equation}\label{eq:1.8}
	\tilde u(x):=F^{-1}\left(G(v(x))\right) \ 
\end{equation}
satisfies
\begin{equation}\label{eq:1.9}
	-\Delta \tilde u=f(\tilde u)
	+\frac{|\nabla{\tilde u}|^2}{f(\tilde u)F(\tilde u)}\Big[ g'(v)G(v)-f'(\tilde u)F(\tilde u)\Big],
\end{equation}
where $G(s):=\displaystyle\int_s^{\infty}\frac{d\eta}{g(\eta)}$.
Let $p_f$ be the  H\"older conjugate of $q_f$ defined 
by 
\[
\frac{1}{p_f}+\frac{1}{q_f}=1.
\]
It should be mentioned that 
if $q_f>1$ then $g(v)=v^{p_f}$ satisfies $g'(s)G(s)=q_f$ for all $s>0$, namely 
the power nonlinearity is a model nonlinearity which satisfies $\eqref{eq:1.7}$. 
Heuristically, the remainder term in \eqref{eq:1.9} 
is no longer affected near the singularity of the solution from \eqref{eq:1.7}
when $g(v) = v^{p_f}$, 
so it is reasonable to assume that 
\eqref{eq:1.9} can be used to investigate 
the singular solution of \eqref{eq:1.1}.
Therefore, via 
the transformation~\eqref{eq:1.8} and
the transformed equation~\eqref{eq:1.9},
we consider relating the general nonlinear problem with a general nonlinearity $f$
to the nonlinear problem with a pure power nonlinearity which is a model case. 
Note that 
the function $v_{p_f}$ defined in \eqref{eq:1.4} with $p=p_f$
satisfies $-\Delta v=v^{p_f}$ and 
\begin{equation}
  \label{eq:1.10}
  G(v_{p_f}(x))=\int_{v_p}^{\infty}\frac{ds}{s^{p_f}}=(q_f-1)v^{-\frac{1}{q_f-1}}=(p_f-1)\left(L_{p_f}|x|^{-\frac{2}{p_f-1}}\right)^{-\frac{1}{p_f-1}}
  =\frac{|x|^2}{2N-4q_f}.
\end{equation}
In conclusion, 
one can expect that 
\begin{equation}
  \label{eq:1.11}
  \tilde u(x):=F^{-1}\left(G(v_{p_f}(x))\right)
  =F^{-1}\left(\frac{|x|^2}{2N-4q_f}\right)
\end{equation}
is an approximate singular solution of $-\Delta u=f(u)$
for general nonlinearities $f$ which satisfy \eqref{eq:1.7}.
Indeed, 
Miyamoto~\cite{M} constructed a singular solution of \eqref{eq:1.1} with $f\in C^2([0,\infty))$ satisfying
$p_\mathrm{S}<p_f$ in the form of
\[
  u(x)=F^{-1}\left(\frac{|x|^2}{2N-4q_f}(1+o(1))\right)\ \ \text{as}\ \ x\to 0.
\]
Later, Miyamoto-Naito~\cite{MN} proved 
that the singular solution constructed in \cite{M} is 
the unique radially symmetric singular solution of \eqref{eq:1.1}.
Since $f(u)=u^p$ is a model nonlinearity satisfying
$p_f=p$, 
their uniqueness result  is a fully generalization of \cite{SZ}.
See \cites{FI,FI2,FI3,FI4,FIRT,FHIL,M,MN} for more developments on elliptic and parabolic problems with general superlinear source terms.

In this paper, we focus on the region $p_\mathrm{c}<p_f<p_\mathrm{S}$ and construct infinitely many singular solution 
for general 
nonlinearity 
satisfying \eqref{eq:1.7}.
To state our result, we further assume 
that L'H\^{o}pital form of the limit \eqref{eq:1.7} which is defined by 
\begin{equation}
  \label{eq:1.12}
  q_f := \lim_{u\to\infty} \frac{f'(u)^2}{f(u)f''(u)}
\end{equation}
exists.
It is pointed out by Miyamoto~\cite{M}
that
the limit \eqref{eq:1.12} is originally introduced by Dupagne-Farina~\cite{DF}, 
and if \eqref{eq:1.12} is satisfied then $f$ also satisfies \eqref{eq:1.7} as
\[
\lim_{u\to\infty} f'(u)F(u)
=
\lim_{u\to\infty} \frac{F(u)}{1/f'(u)}
=
\lim_{u\to\infty} \frac{f'(u)^2}{f(u)f''(u)}=q_f
\]
by L'H\^{o}pital rule. 
Let 
\[
  a:=\frac{4}{p_f-1}-N+2,\qquad b=2N-4q_f.
\]
If $p_\mathrm{c}<p_f<p_\mathrm{S}$, then 
the quadratic equation 
\begin{equation}
  \label{eq:1.13}
  \lambda^2 +a\lambda+b = 0
\end{equation}
has roots whose real parts are negative. 
Let $p_*>1$ be 
\begin{equation}
  \label{eq:1.14} 
  p_* := 1+\frac{4}{N-4+2\sqrt{N-1}}, 
\end{equation}
which satisfies $p_*\in (p_\mathrm{c},p_\mathrm{S})$ for $N\ge 3$. 
When $p_\mathrm{c}<p_f<p_*$, 
equation~\eqref{eq:1.13} has two different 
negative roots $-\lambda_2<-\lambda_1<0$.
When $p_f=p_*$, equation~\eqref{eq:1.13} has the unique 
negative root $-\lambda_*$. 
When $p_*<p_f<p_\mathrm{S}$, equation~\eqref{eq:1.13} has two 
roots that are complex conjugate to each other, 
and they are given by $-\frac{a}{2} \pm ki$ for $k>0$. 
Set 
\begin{equation}
  \label{eq:1.15}
  \displaystyle 
  P(r,s)
  :=
  \left\{
    \begin{aligned}
  &
  \left(\frac{r}{s}\right)^{\lambda_1}
      && \mbox{if} \,\,\, p_\mathrm{c}<p_f<p_*, 
      \\
  &
  \left(1+\log\frac{s}{r}\right)\left(\frac{r}{s}\right)^{\lambda_*}
      && \mbox{if} \,\,\, p_f=p_*, 
      \\
  &
  \left(\frac{r}{s}\right)^{\frac{a}{2}}
      && \mbox{if} \,\,\, p_*<p_f<p_\mathrm{S}.
    \end{aligned}
    \right.
\end{equation}
The following theorem shows us that equation~\eqref{eq:1.1} admits infinitely many singular solutions. 
\begin{theorem}\label{Theorem:1.1}
  Let $N\ge 3$, $f\in C^2((0,\infty))$ satisfy $f,f'>0$ for $s>0$ and 
  the limit \eqref{eq:1.12} exists 
  and $\tilde{u}$ be the function defined by \eqref{eq:1.11}. 
  If $p_\mathrm{c}<p_f<p_\mathrm{S}$, 
  then there exist 
  $r_0,\epsilon>0$ such that for any $\alpha,\beta>0$ with $\alpha+\beta<\epsilon$,
  equation \eqref{eq:1.1} 
  possesses 
  a radially symmetric singular solution in the form of
  \begin{equation}
    \notag 
    u(x)
    = \tilde{u}(x) \left(1+\theta\bigl(|x|\bigr)\right)
    =
    F^{-1}\left(\frac{|x|^2}{2N-4q_f}\right)\left(1+\theta\bigl(|x|\bigr)\right),
  \end{equation}
  where $\theta\in C^2((0,r_0))\cap C^1([0,r_0])$ satisfies
  $\theta(r_0)=\alpha$, $r_0\theta'(r_0)=-\beta$ and 
  \begin{equation}
    \notag 
\left|\theta\bigl(|x|\bigr)\right| + |x|\left|\theta'\bigl(|x|\bigr)\right|
    =
    O\left(
    \int_{|x|}^{r_0}
    P\bigl(|x|,s\bigr)\left|f'\bigl(\tilde u(s)\bigr)F\bigl(\tilde u(s)\bigr)-q_f\right|
    \frac{ds}{s}
    \right)
  \end{equation}
  as $x\to 0$.
\end{theorem}

As an application of Theorem~\ref{Theorem:1.1}, 
we consider a typical example of the nonlinear term, $f(u) = u^p+u^r$, 
where $p>1$ and $0<r<p$. 
Let $r^*$ be the exponent defined by
\begin{equation}
  \label{eq:1.16}
  r^*
  :=
  \left\{
  \begin{aligned}
  &
  \frac{p-1}{4}\left(N-2-\sqrt{\left(N-2-\frac{4}{p-1}\right)^2-8\left(N-2-\frac{2}{p-1}\right)}\right)
  && \text{if} \ p_\mathrm{c}< p<p_*,
  \\
  &
  \frac{(p-1)(N-2)}{4} 
  && \text{if} \ p_*\le p<p_\mathrm{S}.
  \end{aligned}
  \right.
\end{equation}
Applying Theorem~\ref{Theorem:1.1} to \eqref{eq:1.1} with this nonlinear term, 
we obtain the following result. 
\begin{corollary}
  \label{Corollary:1.1}
  Let $N\ge 3$ and $f(s)=s^p+s^r$ with $p_\mathrm{c}<p<p_\mathrm{S}$ and $0<r<p$. 
  Then the same conclusion as Theorem~$\mathrm{\ref{Theorem:1.1}}$ holds. Furthermore, 
  the upper estimates of $|\theta(|x|)| + |x||\theta'(|x|)|$ as $|x| \to 0$
  are given by a constant multiple of the functions in Table~$\mathrm{\ref{table:1}}$ according to the values of $p$ and $r$.  
  \begin{table}[htbp]
    \centering
    {\renewcommand\arraystretch{2.3}
      \begin{tabular}{|c||c|c|c|}
        \hline 
        $\left|\theta\bigl(|x|\bigr)\right| + |x|\left|\theta'\bigl(|x|\bigr)\right|$
        & $p_\mathrm{c}<p<p_*$ 
        & $p=p_*$ 
        & $p_*<p<p_\mathrm{S}$  
        \\[1mm]
        \hline\hline 
        $\displaystyle 0< r<r^*$ 
        & $\displaystyle |x|^{\lambda_1}$
        & $\displaystyle |x|^{\lambda_*}\log\frac{1}{|x|}$
        &
        $\displaystyle |x|^{\frac{a}{2}}$
        \\[2mm]
        \hline
        $\displaystyle r=r^*$ 
        & $\displaystyle  |x|^{\lambda_1 }\log\frac{1}{|x|}$
        & $\displaystyle  |x|^{\lambda_*}\left(\log\frac{1}{|x|}\right)^2$
        &
        $\displaystyle  |x|^{\frac{a}{2}}\log\frac{1}{|x|}$
        \\[2mm]
        \hline
        \rule[0cm]{-1.5mm}{9mm}
        \raisebox{4pt}{$\displaystyle r^*<r<p$}
        & 
        \multicolumn{3}{c|}{
          \raisebox{4pt}{$\displaystyle |x|^{\frac{2(p-r)}{{p-1}}}$}
        }
        \\
        \hline 
      \end{tabular}
    }
    \caption{Upper estimates of $\left|\theta\bigl(|x|\bigr)\right| + |x|\left|\theta'\bigl(|x|\bigr)\right|$}\label{table:1}
  \end{table}
\end{corollary}
See Section~\ref{section:4}
for further examples, e.g., 
\begin{itemize}
  \item $f(s)=s^p(\log s)^r$ with $p_\mathrm{c}<p<p_\mathrm{S}$ and $r\in \mathbb{R}$. 
  \item $f(s)=s^p\exp((\log s)^r)$ with $p_\mathrm{c}<p<p_\mathrm{S}$ and $0<r<1$. 
  \item $f(s)=s^p+s^r(\log s)^{\beta}$ with $p_\mathrm{c}<p<p_\mathrm{S}$, $0<r<p$ and $\beta \in \mathbb{R}$.
\end{itemize}

\begin{remark}
\label{Remark:1.1}
By examining the asymptotic behavior of
the function $F^{-1}(\sigma)$ as $\sigma \to +0$, 
one can obtain 
the specific divergence rates
of the singular solutions constructed in Theorem~$\mathrm{\ref{Theorem:1.1}}$. 
Since this is not our aim, we only give the computation for $f(s)=s^p+s^r$ in Appendix. 
\end{remark}

We provide a summary of the discussions in this paper.
Keeping in mind that function~\eqref{eq:1.11} is the principal term of singular solutions, 
we show that there exist infinitely many radially symmetric solutions of \eqref{eq:1.1} in the form of 
\begin{equation}
  \label{eq:1.17}
  F^{-1}\left(\frac{|x|^2}{2N-4q_f}\right)(1+\theta(|x|)). 
\end{equation}
To this end, we derive the equation for $\theta$ and show the multiplicity of $\theta$. 
After applying the Emden transformation $\eta(\rho):=\theta(|x|)$
with $\rho:=\log (1/|x|)$ to the equation regarding $\theta$, 
we get a nonlinear equation regarding $\eta$. 
One major difference from the pure power case is the form of the nonlinear equation of $\eta$. 
In the case of pure power nonlinear terms,
the equation of $\eta$, which is eventually an ordinary differential equation, is autonomous, the boundedness of $\eta$ is known and 
the set of the steady states is discrete. 
Therefore, the behavior of $\eta=\eta(\rho)$ as $\rho\to\infty$ can be studied in detail. 
In the present case, however, 
the equation is not autonomous and the boundedness of $\eta$ is unknown, 
so it is generally difficult to investigate the behavior of $\eta$ as $\rho\to\infty$. 
Therefore, in this paper, 
we rewrite the equation as an integral equation and find desired solutions using the fixed point theorem.

The next question is how to find multiple solutions for the equation of $\eta$. 
The key to this point is that the characteristic equation of the linear part of the equation 
is given by \eqref{eq:1.13} (see \eqref{eq:2.11}). 
Thanks to $p_f\in (p_{\mathrm{c}},p_{\mathrm{S}})$, 
the real parts of roots of \eqref{eq:1.13} are negative, 
so the related linear ordinary differential equation has two fundamental solutions 
which decay to zero as $\rho\to \infty$. 
This fact suggests that when converting the nonlinear equation to integral equations, 
the degree of freedom of the fundamental solutions arises (see \eqref{eq:2.15}), 
and since the fact that adding the fundamental solution causes the solution to decay to zero remains unchanged, 
the possibility arises that infinitely many $\eta(\rho)$, i.e., $\theta(|x|)$ which decay to zero may exist.
Since the multiplicity of $\theta$ implies the multiplicity of solutions of the form \eqref{eq:1.17}, 
there will be infinitely many singular solutions of \eqref{eq:1.1}.
The important point is that, since there are many ways to rewrite integral equations, 
we can demonstrate the multiplicity of solutions to the equation of $\eta$ 
by proving the existence of a solution to each integral equation. 

\medskip 

The organization of this paper is as follows. 
Section~\ref{section:2} contains some auxiliary results
to analyze equation~\eqref{eq:1.1} with general superlinear nonlinearity $f(u)$.
In Section~\ref{section:3} we give the proof of Theorem~\ref{Theorem:1.1}.
Finally, in Section~\ref{section:4} we focus on several examples and give the estimate of $\theta(|x|)$ as $x\to 0$.

\section{Preliminaries}
\label{section:2}

In this section we derive the equation for $\theta$ and rewrite the equation 
using the Emden transformation. 
Furthermore, we introduce useful lemmas to prove main results. 

In what follows, we use the following notation. 
For any set $\Sigma$, let $g=g(\sigma)$ and $h=h(\sigma)$ 
be maps from $\Sigma$ to $(0,\infty)$. 
Then we say 
$$
g(\sigma)\lesssim h(\sigma)
$$
for all $\sigma\in\Sigma$
if there exists a positive constant $C$ such that $g(\sigma)\le Ch(\sigma)$ 
for all $\sigma\in\Sigma$. 
In addition, we say 
$$
g(\sigma)\simeq h(\sigma)
$$ 
for all $\sigma\in\Sigma$
if $g(\sigma)\lesssim h(\sigma)$ and $h(\sigma)\lesssim g(\sigma)$ for all $\sigma\in\Sigma$.
Furthermore, since all functions of $x\in\mathbb{R}^N$ 
appearing hereafter are radially symmetric, 
they are assumed to be functions of $r = |x|$ for the sake of brevity.

\subsection{Derivation of the equation for the remainder term}
We derive an equation for $\theta$ and rewrite it using the Emden transformation. 
In the following, we consider only radially symmetric solutions of \eqref{eq:1.1}, 
so we use $u(x)$ and $u(|x|)$ interchangeably. 

Suppose that 
limit \eqref{eq:1.12} exists and 
set 
\begin{equation}
  \label{eq:2.1}
  g(s) := s^{p_f}, 
  \qquad 
  G(s) := \int_s^\infty \frac{1}{g(\tau)}\, d\tau 
  = 
  \frac{1}{p_f-1}s^{-(p_f-1)}. 
\end{equation}
Note that $p_f$ is the H\"{o}lder conjugate of $q_f$. 
Then we have 
\begin{equation}
  \label{eq:2.2}
  g'(s)G(s) = q_f 
  \quad\mbox{for all}\,\,\, s>0. 
\end{equation}
Let $v_{p_f}(r)$ with $r=|x|$    
be a singular solution of \eqref{eq:1.1} with $f(s)=g(s)$, 
defined by \eqref{eq:1.4} with $p=p_f$.  
Then it holds
\begin{equation}\label{eq:2.3}
-v_{p_f}''-\frac{N-1}{r} v_{p_f}=g(v_{p_f}).
\end{equation}
Set 
\begin{equation}
  \label{eq:2.4}
  \tilde u(r) := F^{-1}(G(v_{p_f}(r)))=F^{-1}\left(
    \frac{r^2}{2N-4q_f} 
  \right),
\end{equation}
where the last equality follows from \eqref{eq:1.10}.

In the following, we seek to find a singular solution to \eqref{eq:1.1} of the following form: 
\begin{equation}
  \label{eq:2.5}
  u(r) = 
  \tilde u(r)(1 + \theta(r)) 
  \quad\mbox{with}\quad r=|x|, \,\,\, 
  \theta \in C^2((0,r_0)) \cap C^1((0,r_0]), 
\end{equation}
for some $r_0>0$, 
where $\tilde{u}$ is the function defined by \eqref{eq:1.11}. 
We also define the following function 
\begin{equation}
  \notag 
  \begin{aligned}
    \tilde{I}(r) 
    & := 
    \frac{(\tilde u')^2}{f(\tilde u)F(\tilde u)}\left[
      f'(\tilde u)F(\tilde u) - q_f
    \right], 
  \end{aligned}
\end{equation}
for $r\in (0,r_0)$. 
It follows from \eqref{eq:2.4} that 
$F(\tilde u)=G(v_{p_f})$, and differentiating both sides of this equality yields 
\[
  \frac{\tilde u'}{f(\tilde u)}=\frac{v_{p_f}'}{g(v_{p_f})}.
\]
Therefore, 
\begin{equation}
  \notag 
  \begin{aligned} 
    \tilde u'' 
    &= 
    \frac{f(\tilde u)}{g(v_{p_f})}v_{p_f}''+\frac{f'(\tilde u)g(v_{p_f})\tilde u'-f(\tilde u)g'(v_{p_f})v_{p_f}'}{g(v_{p_f})^2}v_{p_f}'
    = 
    \frac{f(\tilde u)}{g(v_{p_f})}v_{p_f}''+\frac{f'(\tilde u)-g'(v_{p_f})}{f(\tilde u)}(\tilde u')^2 
    \\
    &= 
    \frac{f(\tilde u)}{g(v_{p_f})}v_{p_f}''+\frac{f'(\tilde u)F(\tilde u)-g'(v_{p_f})G(v_{p_f})}{f(\tilde u)F(\tilde u)}(\tilde u')^2.
  \end{aligned}
\end{equation}
Thanks to 
\eqref{eq:2.2} and
\eqref{eq:2.3}, 
we have 
\begin{equation}
  \notag 
  -\tilde u''-\frac{N-1}{r}\tilde u'=f(\tilde u)-\tilde{I}(r).
\end{equation}
Furthermore, since 
\begin{equation}
  \label{eq:2.6}
  \tilde{u}'(r) = -f(\tilde{u}(r))\frac{r}{N-2q_f} = -\frac{2}{r}f(\tilde{u}(r))F(\tilde{u}(r)) 
\end{equation}
and 
\begin{equation}
  \notag
  \begin{aligned}
    -u''-\frac{N-1}{r}u' 
    = 
    \left(
      -\tilde u''-\frac{N-1}{r}\tilde u'
    \right)(1+\theta)
    -\tilde u\theta''-\left(\frac{N-1}{r}\tilde u+2\tilde{u}'\right)\theta', 
  \end{aligned}
\end{equation}
we see that $u$ satisfies \eqref{eq:1.1} in $B_{r_0}\setminus \{0\}$ if and only if 
$\theta$ satisfies
\begin{equation}
  \notag 
  \tilde u \theta'' + \left(\frac{N-1}{r}\tilde u-\frac{4f(\tilde u)F(\tilde u)}{r}\right)\theta' 
  + (\tilde{I}(r) - f(\tilde u))\theta 
  + \tilde{I}(r)
  + f(\tilde u(1+\theta)) - f(\tilde u) 
  = 0
\end{equation}
on $(0,r_0)$. 

Set 
\begin{equation}
  \label{eq:2.7}
  \eta(\rho) := \theta(r) 
  \quad\mbox{and}\quad 
  \phi(\rho) := \tilde u(r)
  \quad\mbox{with}\quad 
  \rho=\log (1/r). 
\end{equation}
Since $\eta'(\rho) = -e^{-\rho}\theta'(r)$ and 
$\eta''(\rho) = e^{-\rho}\theta'(r) + e^{-2\rho}\theta''(r)=-\eta'(\rho)+ e^{-2\rho}\theta''(r)$
and \eqref{eq:2.4} yields $e^{-2\rho} = r^2 = (2N-4q_f)F(\phi)$, 
putting 
\begin{equation}
  \label{eq:2.8}
  I(\rho) := \frac{r^2 \tilde{I}(r)}{\tilde u(r)} 
  = 
  \frac{4f(\phi)F(\phi)}{\phi} \left(
    f'(\phi)F(\phi) - q_f
  \right), 
\end{equation}
we have 
\begin{equation}
  \label{eq:2.9}
  \begin{aligned}
    \eta''+\eta'-\left(
      N-1-\frac{4f(\phi)F(\phi)}{\phi}
    \right)\eta'
    + \left(
      I(\rho) - (2N-4q_f)\frac{f(\phi)F(\phi)}{\phi}
    \right) \eta 
    + I(\rho) 
    \quad 
    \\
    +(2N-4q_f)\frac{F(\phi)}{\phi}\left(
      f(\phi(1+\eta)) - f(\phi)
    \right)
    = 0.
  \end{aligned}
\end{equation}
Then, putting 
\begin{equation}
  \label{eq:2.10}
  \begin{aligned}
    a &:= 
    \frac{4}{p_f-1} - N+2,
    \qquad
    b :=
    2N-4q_f 
    =2\left(N-2-\frac{2}{p_f-1}\right),
    \\
    L_1(\rho)
    & := 
    (2N-4q_f)\left[
      (f'(\phi)F(\phi) - q_f) - 
      \left( \frac{f(\phi)F(\phi)}{\phi} - \frac{1}{p_f-1} \right)
    \right] + I(\rho), 
    \\
    L_2(\rho)
    & := 
    4\left(
      \frac{f(\phi)F(\phi)}{\phi}-\frac{1}{p_f-1}
    \right), 
    \\ 
    N[\eta](\rho)
    & :=
    (2N-4q_f) \frac{F(\phi)}{\phi}\left(
      f(\phi(1+\eta)) - f(\phi) - f'(\phi)\phi\eta 
    \right), 
  \end{aligned}
\end{equation}
by \eqref{eq:2.9} we obtain 
\begin{equation}
  \label{eq:2.11}
  \eta''+a\eta'+b\eta+I(\rho) + L_1(\rho)\eta + L_2(\rho)\eta'+N[\eta](\rho) = 0. 
\end{equation}
In what follows we consider equation~\eqref{eq:2.11} 
and prove that \eqref{eq:2.11} admits infinitely many solutions satisfying $\eta(\rho) \to 0$ as $\rho\to\infty$, 
which turns out to be the multiplicity of singular solutions of \eqref{eq:1.1}. 

Let $p_*>1$ be the exponent defined by \eqref{eq:1.14}
and
$m:=2/(p_f-1)$.
Set 
\begin{equation}
  \label{eq:2.12}
  \begin{gathered}
    \lambda_1 := \frac{2m-N+2-\sqrt{(N-2-2m)^2-8(N-2-m)}}{2}, 
    \\ 
    \lambda_2 := \frac{2m-N+2+\sqrt{(N-2-2m)^2-8(N-2-m)}}{2}, 
  \end{gathered}
\end{equation}
for $p_f\in (p_\mathrm{c},p_{*})$ 
and  
\begin{equation}
  \label{eq:2.12a}
  \lambda_* := 
  \frac{2}{p_*-1}-\frac{N-2}{2} 
\end{equation}
for $p_f =p_*$.
We recall the following facts concerning the roots of equation~\eqref{eq:1.13}. 
When $p_\mathrm{c}<p_f<p_*$, 
equation~\eqref{eq:1.13} has two different 
negative roots $-\lambda_1$ and $-\lambda_2$. 
When $p_f=p_*$, equation~\eqref{eq:1.13} has the unique 
negative root $-\lambda_*$. 
When $p_*<p_f<p_\mathrm{S}$, equation~\eqref{eq:1.13} has two 
roots that are complex conjugate to each other, 
and they are given by $-\frac{a}{2} \pm ki$ for $k>0$. 
It should be mentioned that 
\[
\lambda_1=\lambda_2=\frac{a}{2}=\lambda_*
\]
if $p_f=p_*$.

Let $\Phi_1$ and $\Phi_2$ be the linearly independent solutions of 
\begin{equation}
  \label{eq:2.13}
  \eta''+a\eta'+b\eta=0, 
\end{equation}
defined by 
\begin{equation}
  \label{eq:2.14} 
  (\Phi_1(\rho),\Phi_2(\rho)) = 
  \begin{cases}
    (e^{-\lambda_1\rho}, e^{-\lambda_2\rho}) 
    & \mbox{if} \,\,\, p_\mathrm{c}<p<p_*, 
    \\
    (e^{-\lambda_*\rho}, \rho e^{-\lambda_*\rho})
    & \mbox{if} \,\,\, p=p_*, 
    \\
    (e^{-\frac{a}{2}\rho}\cos (k\rho), e^{-\frac{a}{2}\rho}\sin (k\rho)) 
    & \mbox{if} \,\,\, p_*<p<p_\mathrm{S}. 
  \end{cases}
\end{equation}
By variation of parameters one can rewrite equation \eqref{eq:2.11} as the following integral form:
\begin{equation}
  \label{eq:2.15}
  \begin{aligned}
    \eta(\rho) 
    &= 
    C_1\Phi_1 + C_2\Phi_2 - 
    \int_{\rho_0}^\rho 
    K(\rho,\tau)
    \{I(\tau) + L_1(\tau)\eta(\tau) + L_2(\tau)\eta'(\tau)+N[\eta](\tau)\}\, d\tau, 
  \end{aligned}
\end{equation}
where  $C_1,C_2$ is a constant and
\begin{equation}
  \notag 
  K(\rho,\tau) := 
  \frac{\Phi_1(\tau)\Phi_2(\rho)-\Phi_1(\rho)\Phi_2(\tau)}{
    W(\tau) 
  }, 
  \qquad 
  W(\tau) := \Phi_1(\tau)\Phi_2'(\tau) - \Phi_1'(\tau)\Phi_2(\tau). 
\end{equation}
Then we have 
\begin{equation}
  \notag 
  W(\tau) = 
  \begin{cases}
    -(\lambda_2-\lambda_1)e^{-(\lambda_1+\lambda_2)\tau} 
    & \mbox{if} \,\,\, p_\mathrm{c}<p_f<p_*, 
    \\
    e^{-2\lambda_*\tau}
    & \mbox{if} \,\,\, p_f=p_*, 
    \\
    ke^{-a\tau} 
    & \mbox{if} \,\,\, p_*<p_f<p_\mathrm{S}, 
  \end{cases}
\end{equation}
and obtain 
\begin{equation}
  \label{eq:2.16}
  K(\rho,\tau) = 
  \begin{cases}
    \dfrac{1}{\lambda_2-\lambda_1}
    \left(
      e^{-\lambda_1(\rho-\tau)}-e^{-\lambda_2(\rho-\tau)}
    \right)
    & \mbox{if} \,\,\, p_\mathrm{c}<p_f<p_*, 
    \\[7pt]
    (\rho-\tau) e^{-\lambda_*(\rho-\tau)}
    & \mbox{if} \,\,\, p_f=p_*, 
    \\[5pt] 
    \dfrac{1}{k}\sin (k(\rho-\tau)) e^{-\frac{a}{2}(\rho-\tau)}
    & \mbox{if} \,\,\, p_*<p_f<p_\mathrm{S}. 
  \end{cases}
\end{equation}
A direct computation shows that 
the derivative of $K$ with respect to $\rho$ is given by 
\begin{equation}
  \label{eq:2.17}
  \partial_\rho K(\rho,\tau) = 
  \begin{cases}
    \dfrac{1}{\lambda_2-\lambda_1}
    \left(
      -\lambda_1 e^{-\lambda_1(\rho-\tau)}+\lambda_2 e^{-\lambda_2(\rho-\tau)}
    \right)
    & \mbox{if} \,\,\, p_\mathrm{c}<p_f<p_*, 
    \\[7pt]
    (1-\lambda_* (\rho-\tau)) e^{-\lambda_*(\rho-\tau)}
    & \mbox{if} \,\,\, p_f=p_*, 
    \\[5pt] 
    \left(
      \cos (k(\rho-\tau))-\dfrac{a}{2k}\sin (k(\rho-\tau)) 
    \right) e^{-\frac{a}{2}(\rho-\tau)}
    & \mbox{if} \,\,\, p_*<p_f<p_\mathrm{S}. 
  \end{cases}
\end{equation}

\subsection{Useful lemmas}
We prepare several useful lemmas to solve the integral equation~\eqref{eq:2.15}.
\begin{lemma}
  \label{Lemma:2.1}
  Assume that limit $q_f$ given by \eqref{eq:1.12} exists and $q_f>1$, and let $p_f$ be the 
  H\"older conjugate of $q_f$.
  Let $\tilde u$ be the function defined by \eqref{eq:2.4} and put $\phi(\rho) = \tilde u(r)$ with $\rho=\log (1/r)$. 
  Then there hold the following$\mathrm{:}$ 
  \begin{itemize}
    \item[$\mathrm{(i)}$] 
    \begin{math}
      \displaystyle 
      \lim_{s\to \infty} \frac{f(s)F(s)}{s} = \frac{1}{p_f-1}.
    \end{math}
    In particular, $\displaystyle\lim_{\rho\to \infty} \frac{f(\phi(\rho))F(\phi(\rho))}{\phi(\rho)} = \frac{1}{p_f-1}$.
    \item[$\mathrm{(ii)}$]
    \begin{math}
      \displaystyle
      \lim_{\rho\to\infty} \frac{\phi'(\rho)}{\phi(\rho)} = \frac{2}{p_f-1}. 
    \end{math}
    \item[$\mathrm{(iii)}$]
    $\lim\limits_{\rho\to\infty} I(\rho) = 0$. 
    \item[\(\mathrm{(iv)}\)] 
    \(
      \displaystyle \lim_{\rho\to\infty} I'(\rho) = 0. 
    \)  
  \end{itemize}
\end{lemma}

\begin{proof}
  Since \(\{f(s)F(s)\}'=f'(s)F(s)-1\to q_f-1=1/(p_f-1)>0\) as \(s\to\infty\), 
  we have \(f(s)F(s)\to\infty\) as \(s\to\infty\). 
  Then, by L'H\^{o}pital rule we have  
  \begin{equation}
    \notag 
    \lim_{s\to\infty}\frac{f(s)F(s)}{s} = \lim_{s\to\infty} (f'(s)F(s)-1) = \frac{1}{p_f-1}. 
  \end{equation}
  Thus assertion~(i) follows. 
  Note that the latter of assertion~(i) follows from \(\phi(\rho)\to\infty\) as \(\rho\to\infty\).
  Since $\phi(\rho) = \tilde{u}(r)$ with $\rho=\log (1/r)$, 
  by \eqref{eq:2.6} we have
  \begin{equation}
    \label{eq:2.18}
    \phi'(\rho) = 2f(\phi(\rho))F(\phi(\rho)). 
  \end{equation} 
  Then we see that assertion~(ii) follows from assertion~(i). 
  Assertion~(iii) follows from \eqref{eq:1.12} and assertion~(i). 

  We prove the assertion~(iv). 
  Since it follows from \eqref{eq:2.8} that 
  \begin{equation}
    \notag 
    \frac{I'(\rho)}{4} =
    \frac{f'(\phi)F(\phi)-1-\frac{f(\phi)F(\phi)}{\phi}}{\phi}
    (f'(\phi)F(\phi) - q_f)\phi'
    + 
    \frac{f(\phi)F(\phi)}{\phi} \left(
      f''(\phi)F(\phi) - \frac{f'(\phi)}{f(\phi)}
    \right)\phi', 
  \end{equation}
  by \eqref{eq:1.7} and assertions~(i) and (ii) we see that it suffices to prove 
  \begin{equation}
    \label{eq:2.19}
    \lim_{\rho\to\infty}
    \left(
      f''(\phi)F(\phi) - \frac{f'(\phi)}{f(\phi)}
    \right) \phi'(\rho)
    = 0.
  \end{equation}
  By \eqref{eq:2.18} we obtain 
  \begin{equation}
    \notag 
    \left(
      f''(\phi)F(\phi) - \frac{f'(\phi)}{f(\phi)}
    \right) \phi'(\rho)
    = 
    2\left(
      \frac{f''(\phi)f(\phi)}{(f'(\phi))^2} (f'(\phi)F(\phi))^2 - f'(\phi)F(\phi)
    \right). 
  \end{equation}
  This together with \eqref{eq:1.7} and \eqref{eq:1.12} yields \eqref{eq:2.19}. 
  Thus assertion~(iv) follows. 
\end{proof}

\begin{lemma}\label{Lemma:2.2}
  Suppose that the hypotheses of Lemma~$\mathrm{\ref{Lemma:2.1}}$ hold.
  Then the following hold for any $\lambda>0\mathrm{:}$
  \begin{enumerate}[\rm(i)]
    \item
    $\displaystyle \lim_{\tau \to \infty}
    \int_{\rho_0}^{\tau}(1+(\tau-t))e^{-\lambda(\tau-t)}|I(t)|\, dt=0$.
    \item
    $\displaystyle 
    \lim_{\rho_0\to \infty}
    \left(
    \sup_{\tau \ge \rho_0}\int_{\rho_0}^{\tau}(1+(\tau-t))e^{-\lambda(\tau-t)}|I(t)|\, dt
    \right)
    =0$.
    \item
    $\displaystyle \lim_{\tau \to \infty}
    \int_{\rho_0}^{\tau}e^{-\lambda(\tau-t)}|I(t)|\, dt=0$.
    \item
    $\displaystyle 
    \lim_{\rho_0\to \infty}
    \left(
    \sup_{\tau \ge \rho_0}\int_{\rho_0}^{\tau}e^{-\lambda(\tau-t)}|I(t)|\,dt
    \right)
    =0$.
  \end{enumerate}
\end{lemma}
\begin{proof}[Proof of Lemma~$\mathrm{\ref{Lemma:2.2}}$]
  Assertion (i) is obvious if 
  $\int_{\rho_0}^{\infty}e^{\lambda \tau}|I(\tau)| \, d\tau <\infty$,
  so we consider the case 
  of 
  $\int_{\rho_0}^{\infty}e^{\lambda \tau}|I(\tau)| \, d\tau =\infty$.
  In this case,  
  L'H\^{o}pital's rule and Lemma~\ref{Lemma:2.1}~(iii) 
  show us that
  \[
  \begin{aligned}
    \lim_{\tau \to \infty}
    \frac{1}{e^{\lambda \tau}}\int_{\rho_0}^{\tau}(1+(\tau-t))e^{\lambda t}|I(t)|\,dt
    & =
    \lim_{\tau \to \infty}
    \frac{1}{\lambda e^{\lambda \tau}}\left(
      e^{\lambda \tau}|I(\tau)|+\int_{\rho_0}^{\tau}e^{\lambda t}|I(t)|\,dt
    \right)
    \\ 
    & = 
    \lim_{\tau\to \infty} \frac{1}{\lambda}|I(\tau)| + 
    \lim_{\tau\to \infty} \frac{e^{\lambda\tau} |I(\tau)|}{\lambda^2 e^{\lambda \tau}}
    = 0.
  \end{aligned}
  \]
  This completes the proof of (i). 

  More generally, we show that the following property is true, thereby proving assertion~(ii): 
  if a non-negative function $h$ satisfies
  \begin{equation}
    \notag 
    \lim_{\tau\to \infty}\int_{\rho_0}^\tau h(\tau,t)dt= 0, 
  \end{equation}
  then
  \[
    \lim_{\rho_0\to \infty}
    \left(
    \sup_{\tau \ge \rho_0}\int_{\rho_0}^{\tau}
    h(\tau,t)\, dt
    \right)
    =0.
  \]
  Fix arbitrary $\epsilon>0$. Then there exists $\tau_0>\rho_0$ such that 
  $\displaystyle\int_{\rho_0}^{\tau}h(\tau,t)\, dt<\epsilon$ for all $\tau>\tau_0$. 
  Let $\tau_1\ge \tau_0$. It follows from $(\tau_1,\tau)\subset (\tau_0,\tau)$ that
  \[
  \int_{\tau_1}^{\tau}h(\tau,t)\, dt 
  \le 
  \int_{\tau_0}^{\tau}h(\tau,t)\, dt 
  < \epsilon
  \]
  for all $\tau\ge \tau_1$.
  Therefore, it holds
  \[
  \sup_{\tau\ge \tau_1}
  \int_{\tau_1}^{\tau}h(\tau,t)\, dt 
  \le \epsilon.
  \]
  This implies that assertion~(ii) is true.
  Assertions (iii) and (iv) follows from assertions (i) and (ii), respectively.
\end{proof}

We prepare estimates for the nonlinear term in \eqref{eq:2.15},
which will be used to apply the fixed point theorem 
to find a solution of \eqref{eq:2.15}.

\begin{lemma}[Nonlinear estimate]
  \label{Lemma:2.3}
  Suppose that the hypotheses of Lemma~$\mathrm{\ref{Lemma:2.1}}$ hold.
  Then there exist constants $\rho_0>0$ and $C>0$ such that 
  \begin{equation}
    \label{eq:2.22}
    |N[\eta_1](\rho)-N[\eta_2](\rho)| 
    \le 
    C (|\eta_1(\rho)|+|\eta_2(\rho)|) |\eta_1(\rho)-\eta_2(\rho)| 
  \end{equation}
  for all $\rho\ge \rho_0$ and any $\eta_i$ $(i=1,2)$ satisfying $\eta_i(\rho)\to 0$ as $\rho\to\infty$. 
  Here $C$ depends only on $N$ and $q_f$.
\end{lemma}

\begin{proof}
  We first prove 
  that
  \begin{equation}\label{eq:2.23}
  \lim_{\rho\to \infty}\frac{F(\phi+\eta\phi)}{F(\phi)}=1
  \end{equation}
  for any $\eta=\eta(\rho)$ satisfying $\eta(\rho)\to 0$ as $\rho\to\infty$. 
  Indeed, by the mean value theorem we have
  \begin{equation}\label{eq:2.24}
  F(\phi+\phi|\eta|)=F(\phi)-\frac{1}{f(\phi+c_1\phi|\eta|)}\phi|\eta|\ge F(\phi)-\frac{\phi|\eta|}{f(\phi)},
  \end{equation}
  \begin{equation}\label{eq:2.25}
  F(\phi-\phi|\eta|)=F(\phi)+\frac{1}{f(\phi-c_2\phi|\eta|)}\phi|\eta|\le F(\phi)+\frac{\phi|\eta|}{f(\phi-\phi|\eta|)},
  \end{equation}
  for some $c_1$, $c_2\in (0,1)$.
  By \eqref{eq:2.24} and the monotonicity of $F$, it holds that
  \[
  1\ge \frac{F(\phi+\phi|\eta|)}{F(\phi)} 
  \ge 1-\frac{1}{f(\phi)F(\phi)}\phi|\eta|.
  \]
  Since $\eta(\rho) \to 0$ as $\rho\to \infty$, we obtain
  $\lim\limits_{\rho\to \infty}\dfrac{F(\phi+\phi|\eta|)}{F(\phi)}=1$.
  Similarly, by \eqref{eq:2.25} and the monotonicity of $F$ we see that
  \[
  1\ge \frac{F(\phi)}{F(\phi-\phi|\eta|)}\ge 1-\frac{\phi-\phi|\eta|}{f(\phi-\phi|\eta|)F(\phi-\phi|\eta|)}\frac{|\eta|}{1-|\eta|}.
  \]
  Letting $\rho \to \infty$ yields 
  $\lim\limits_{\rho\to \infty}\dfrac{F(\phi)}{F(\phi-\phi|\eta|)}=1$. 
  These together with the fact that 
  \[
  \frac{F(\phi+\phi|\eta|)}{F(\phi)}\le \frac{F(\phi+\phi\eta)}{F(\phi)}\le \frac{F(\phi-\phi|\eta|)}{F(\phi)}
  \]
  implies \eqref{eq:2.23}.

  Let us prove Lemma~\ref{Lemma:2.3} with the aid of \eqref{eq:2.23}.
  The mean value theorem implies that there exist $c,d\in (0,1)$ such that
  \[
  \begin{aligned}
&
  N[\eta_1]-N[\eta_2]
  \\
  &
  =
  (2N-4q_f)\frac{F(\phi)}{\phi}
  [f(\phi+\phi\eta_1)-f(\phi+\phi\eta_2)-f'(\phi)(\eta_1-\eta_2)]
  \\
  &
  =
  (2N-4q_f)\frac{F(\phi)}{\phi}
  [f'(\phi+(c\eta_1+(1-c)\eta_2)\phi)-f'(\phi)]\phi(\eta_1-\eta_2)
  \\
  &
  =
  (2N-4q_f)
  \frac{F(\phi)}{\phi}
  f''(\phi+d(c\eta_1+(1-c)\eta_2)\phi)\phi^2(\eta_1-\eta_2)(c\eta_1+(1-c)\eta_2).
  \end{aligned}
  \]
  Setting
  \(
  \phi_{c,d}:=\phi+d(c\eta_1+(1-c)\eta_2)\phi, 
  \)
  we have 
  \[
  |N[\eta_1]-N[\eta_2]|
  \le C_0 
  \frac{F(\phi)}{\phi}
  |f''(\phi_{c,d})|
  \phi^2(|\eta_1|+|\eta_2|)|\eta_1-\eta_2|, 
  \]
  where $C_0>0$ is a constant. 
  Since ${\phi_{c,d}}/{\phi}\to 1$ as $\rho \to \infty$, 
  it follows from 
  \eqref{eq:1.12}, \eqref{eq:2.23} and Lemma~\ref{Lemma:2.1}~(i) that
  \[
  \begin{aligned}
  \frac{F(\phi)}{\phi}
  |f''(\phi_{c,d})|
  \phi^2
  &
  =
  \frac{F(\phi)}{F(\phi_{c,d})}\frac{\phi_{c,d}}{f(\phi_{c,d})F(\phi_{c,d})}\frac{\phi}{\phi_{c,d}}
\left|
\frac{f''(\phi_{c,d})f(\phi_{c,d})}{f'(\phi_{c,d})^2}
\right|
(f'(\phi_{c,d})F(\phi_{c,d}))^2
  \to 
  p_f
 \end{aligned}
  \]
  as $\rho \to \infty$, 
  yielding \eqref{eq:2.22}. 
\end{proof}

\begin{lemma}\label{Lemma:2.4}
  Suppose that the hypotheses of Lemma~$\mathrm{\ref{Lemma:2.1}}$ hold.
  Let $\lambda\in (0,2/(p_f-1))$. 
  Then there exist constants $\rho_0>0$ and $C>0$ such that 
  \begin{equation}\label{eq:2.26}
    \begin{aligned}
      &
      \int_{\rho_0}^{\rho}e^{\lambda\tau}\left|
        \frac{f(\phi(\tau))F(\phi(\tau))}{\phi(\tau)}-\frac{1}{p_f-1}
      \right| d\tau
      \\
      &
      \le 
      C
      \left(
        \int_{\rho_0}^{\rho}e^{\lambda\tau}|f'(\phi(\tau))F(\phi(\tau))-q_f|\, d\tau 
        +
        e^{\lambda\rho_0}\left|\frac{f(\phi(\rho_0))F(\phi(\rho_0))}{\phi(\rho_0)}-\frac{1}{p_f-1}\right|
      \right)
    \end{aligned}
  \end{equation}
  and
  \begin{equation}\label{eq:2.27}
    \begin{aligned}
      &
      \int_{\rho_0}^{\rho}
      (1+(\rho-\tau))
      e^{\lambda\tau}\left|\frac{f(\phi(\tau))F(\phi(\tau))}{\phi(\tau)}-\frac{1}{p_f-1}\right| d\tau
      \\
      &
      \qquad
      \le 
      C
      \biggl(
        \int_{\rho_0}^{\rho}
        (1+(\rho-\tau))
        e^{\lambda \tau}|f'(\phi(\tau))F(\phi(\tau))-q_f|\, d\tau 
        \\
        &\qquad \qquad
        +
        (1+(\rho-\rho_0))
        e^{\lambda\rho_0}\left|\frac{f(\phi(\rho_0))F(\phi(\rho_0))}{\phi(\rho_0)}-\frac{1}{p_f-1}\right|
      \biggr)
    \end{aligned}
  \end{equation}
  for all $\rho\ge \rho_0$.
  Here $C$ depends only on $N$ and $q_f$.
\end{lemma} 

\begin{proof}[Proof of Lemma~$\mathrm{\ref{Lemma:2.4}}$]
  Let $\rho_0>0$ be a sufficiently large number to be chosen later. 
  It follows from the fundamental theorem of calculus that
  \begin{equation}\label{eq:2.28}
  \begin{aligned}
  &
  \left|f(\phi(\tau))F(\phi(\tau))-\frac{\phi(\tau)}{p_f-1}\right|
  \\
  &
  =
  \left|\int_{\rho_0}^{\tau}\left(f'(\phi(t))F(\phi(t))-1-\frac{1}{p_f-1}\right)\phi'(t)\, dt
  +
  f(\phi(\rho_0))F(\phi(\rho_0))-\frac{\phi(\rho_0)}{p_f-1}
  \right|
  \\
  &
  \le
  \int_{\rho_0}^{\tau}|f'(\phi(t))F(\phi(t))-q_f|
  |\phi'(t)|\, dt
  +
  C_{\rho_0},
  \end{aligned}
  \end{equation}
  where 
  \begin{equation}
    \label{eq:2.29}
    C_{\rho_0}:=
    \left|
    f(\phi(\rho_0))F(\phi(\rho_0))-\frac{\phi(\rho_0)}{p_f-1}
    \right|.
  \end{equation}
  Multiplying both sides of the above inequality by $e^{\lambda\tau}/\phi(\tau)$ 
  and integrating over $[\rho_0,\rho]$,
  we have
  \begin{equation}
    \label{eq:2.30}
    \begin{aligned}
      &
      \int_{\rho_0}^{\rho}e^{\lambda\tau}\left|\frac{f(\phi(\tau))F(\phi(\tau))}{\phi(\tau)}-\frac{1}{p-1}\right|d\tau
      \\
      &
      \le
      \int_{\rho_0}^{\rho}
      \frac{e^{\lambda \tau}}{\phi(\tau)}
      \left(
        \int_{\rho_0}^{\tau}|f'(\phi(t))F(\phi(t))-q||\phi'(t)|\, dt 
      \right) d\tau
      +
      \int_{\rho_0}^{\rho}
      C_{\rho_0} \frac{e^{\lambda \tau}}{\phi(\tau)}
      d\tau
    \end{aligned}
  \end{equation}
  for all $\rho\ge \rho_0$. 
  Since
  \[
  \left(\frac{e^{\lambda\tau}}{\phi(\tau)}\right)'
  =\left(\lambda-\frac{\phi'(\tau)}{\phi(\tau)}\right)\frac{e^{\lambda \tau}}{\phi(\tau)},
  \]
  by Lemma~\ref{Lemma:2.1}~(ii) and $\lambda<\frac{2}{p_f-1}$, it holds
  \begin{equation}\label{eq:2.31}
  \frac{e^{\lambda \tau}}{\phi(\tau)}
  \lesssim
  -\left(\frac{e^{\lambda\tau}}{\phi(\tau)}\right)'
  \qquad 
  \text{and}
  \qquad 
  \frac{|\phi'(\tau)|}{\phi(\tau)}\simeq 1
  \end{equation}
  for sufficiently large $\tau$.
  Hence we have
  \[
    \begin{aligned}
      &
      \int_{\rho_0}^{\rho}
      \frac{e^{\lambda \tau}}{\phi(\tau)}
      \left(
      \int_{\rho_0}^{\tau}|f'(\phi(t))F(\phi(t))-q_f||\phi'(t)|\, dt
      \right) d\tau
      \\
      &
      \lesssim
      -
      \int_{\rho_0}^{\rho}
      \left(\frac{e^{\lambda \tau}}{\phi(\tau)}\right)'
      \left( 
      \int_{\rho_0}^{\tau}|f'(\phi(t))F(\phi(t))-q_f||\phi'(t)|\, dt
      \right) d\tau
      \\
      &
      =
      -
      \frac{e^{\lambda \rho}}{\phi(\rho)}
      \int_{\rho_0}^{\rho}|f'(\phi(t))F(\phi(t))-q_f||\phi'(t)|\, dt
      +
      \int_{\rho_0}^{\rho}
      \frac{e^{\lambda \tau}}{\phi(\tau)}
      |f'(\phi(\tau))F(\phi(\tau))-q_f||\phi'(\tau)|\, d\tau
      \\
      &
      \lesssim
      \int_{\rho_0}^{\rho}
      e^{\lambda \tau}
      |f'(\phi(\tau))F(\phi(\tau))-q_f|\, d\tau,
    \end{aligned}
  \]
  if $\rho_0$ is sufficiently large. 
  In the final inequality we applied Lemma~\ref{Lemma:2.1} (ii) again.
  Meanwhile, 
  \begin{equation}
    \label{eq:2.32}
    \begin{aligned}
      \int_{\rho_0}^{\rho}
      C_{\rho_0} \frac{e^{\lambda \tau}}{\phi(\tau)}\, 
      d\tau
      &
      \lesssim
      -
      \int_{\rho_0}^{\rho}
      C_{\rho_0} \left(\frac{e^{\lambda \tau}}{\phi(\tau)}\right)'
      d\tau
      \lesssim
      C_{\rho_0}\frac{e^{\lambda \rho_0}}{\phi(\rho_0)}
    \end{aligned}
  \end{equation}
  which together with \eqref{eq:2.29} and \eqref{eq:2.30} yields \eqref{eq:2.26}. 

  The proof of \eqref{eq:2.27} is the same as above, using integration by parts. 
  Multiplying both sides of inequality~\eqref{eq:2.28} by $(1+(\rho-\tau))e^{\lambda\tau}/\phi(\tau)$ 
  and integrating over $[\rho_0,\rho]$, we have
  \begin{equation}
    \label{eq:2.33}
    \begin{aligned}
      &
      \int_{\rho_0}^{\rho}
      (1+(\rho-\tau))
      e^{\lambda\tau}\left|\frac{f(\phi(\tau))F(\phi(\tau))}{\phi(\tau)}-\frac{1}{p_f-1}\right|d\tau
      \\
      &
      \le
      \int_{\rho_0}^{\rho}
      (1+(\rho-\tau))
      \frac{e^{\lambda \tau}}{\phi(\tau)}
      \left(
      \int_{\rho_0}^{\tau}|f'(\phi(t))F(\phi(t))-q||\phi'(t)|\,dt
      \right) d\tau
      \\
      &
      \qquad 
      +
      C_{\rho_0} \int_{\rho_0}^{\rho}
      (1+(\rho-\tau))
      \frac{e^{\lambda \tau}}{\phi(\tau)}\, 
      d\tau.
    \end{aligned}
  \end{equation}
  It follows from \eqref{eq:2.28}, \eqref{eq:2.31} and integration by parts that
  \[
    \begin{aligned}
      &
      \int_{\rho_0}^{\rho}
      (1+(\rho-\tau))
      \frac{e^{\lambda \tau}}{\phi(\tau)}
      \left( 
        \int_{\rho_0}^{\tau}|f'(\phi(t))F(\phi(t))-q_f||\phi'(t)|\, dt
      \right) d\tau
      \\
      &
      \lesssim
      -
      \int_{\rho_0}^{\rho}
      \left(\frac{e^{\lambda \tau}}{\phi(\tau)}\right)'
      (1+(\rho-\tau))
      \left(
        \int_{\rho_0}^{\tau}|f'(\phi(t))F(\phi(t))-q_f||\phi'(t)|\, dt
      \right)d\tau
      \\
      &
      =
      -
      \frac{e^{\lambda \rho}}{\phi(\rho)}
      \int_{\rho_0}^{\rho}|f'(\phi(t))F(\phi(t))-q_f||\phi'(t)|\,dt
      \\
      &
      \qquad
      -
      \int_{\rho_0}^{\rho}
      \frac{e^{\lambda \tau}}{\phi(\tau)}
      \left(
        \int_{\rho_0}^{\tau}
        |f'(\phi(t))F(\phi(t))-q_f||\phi'(t)|\, dt
      \right)
      d\tau
      \\
      &
      \qquad
      +
      \int_{\rho_0}^{\rho}
      (1+(\rho-\tau))
      \frac{e^{\lambda \tau}}{\phi(\tau)}
      |f'(\phi(\tau))F(\phi(\tau))-q_f||\phi'(\tau)|\, d\tau
      \\
      &
      \lesssim
      \int_{\rho_0}^{\rho}
      (1+(\rho-\tau))
      e^{\lambda \tau}
      |f'(\phi(\tau))F(\phi(\tau))-q_f|\, d\tau.
    \end{aligned}
  \]
  Meanwhile, it follows from \eqref{eq:2.32} that 
  \[
    \begin{aligned}
      \int_{\rho_0}^{\rho}
      C_{\rho_0} 
      (1+(\rho-\tau))
      \frac{e^{\lambda \tau}}{\phi(\tau)}\, 
      d\tau
      &
      \le 
      C_{\rho_0}
      (1+(\rho-\rho_0))
      \int_{\rho_0}^{\rho}
      \frac{e^{\lambda \tau}}{\phi(\tau)}\, 
      d\tau
      \lesssim 
      C_{\rho_0}
      (1+(\rho-\rho_0))
      \frac{e^{\lambda \rho_0}}{\phi(\rho_0)}, 
    \end{aligned}
  \]
  which together with \eqref{eq:2.29} and \eqref{eq:2.33} yields \eqref{eq:2.27}.
\end{proof}

\section{Proof of Theorem~\ref{Theorem:1.1}}
\label{section:3}

We prove Theorem~\ref{Theorem:1.1} in this section. 
Let us recall the symbols used in Section~\ref{section:2}. 
$\lambda_1$, $\lambda_2$ and $\lambda_*$ are constants defined by \eqref{eq:2.12} and \eqref{eq:2.12a}. 
Let $\rho_0>0$. 
Defining 
\[
  Q(\rho,\tau):=
  \left\{
    \begin{aligned}
      &
      e^{-\lambda_1(\rho-\tau)} && \text{if}\ \ p_\mathrm{c}<p_f<p_*, 
      \\
      &
      (1+(\rho-\tau))e^{-\lambda_*(\rho-\tau)} && \text{if}\ \ p_f=p_*, 
      \\
      &
      e^{-\frac{a}{2}(\rho-\tau)} && \text{if}\ \ p_*<p_f<p_\mathrm{S}, 
    \end{aligned}
  \right.
\]
we see that 
\[
  Q\left(\log\frac{1}{r},\log\frac{1}{s}\right)=P(r,s),
\]
where $P(r,s)$ is defined by \eqref{eq:1.15}. 
Clearly it follows from \eqref{eq:2.16} and \eqref{eq:2.17} that
\begin{equation}
  \label{eq:3.1}
  |K(\rho,\tau)|+\left|\frac{\partial}{\partial \rho}K(\rho,\tau)\right|
  \lesssim  Q(\rho,\tau)
\end{equation}
for all $\rho$, $\tau\ge \rho_0$.  
We redraw Theorem~\ref{Theorem:1.1} using the variables after the Emden transformation, 
and prove the following theorem. 
In particular, we show that 
there exists a solution of equation~\eqref{eq:2.15} 
using the fixed point theorem. 
\begin{theorem}
  \label{Theorem:3.1}
  Assume $p_\mathrm{c}<p_f<p_\mathrm{S}$.
  Then there exist $\rho_0,\epsilon>0$ 
  such that
  for all $\alpha,\beta>0$ with $\alpha+\beta<\epsilon$,
  equation~\eqref{eq:2.15} possesses 
  a solution 
  $\eta\in C^2((\rho_0,\infty))\cap C^1([\rho_0,\infty))$  
  satisfying
  $\eta(\rho_0)=\alpha$, 
  $\eta'(\rho_0)=\beta$ and
  \[
  |\eta(\rho)|+|\eta'(\rho)|
  =O
  \left(
    Q(\rho,\rho_0) + 
    \int_{\rho_0}^{\rho}Q(\rho,\tau)|I(\tau)|\, d\tau
  \right)
  \]
  as $\rho \to \infty$.
\end{theorem}

\begin{remark}\label{Remark:3.1}
  Set 
  \begin{equation}
  \label{eq:3.2}
  \Lambda
  :=
  \left\{
    \begin{aligned}
      & \lambda_1 &&\text{if} \ \ p_\mathrm{c}<p_f<p_*,
      \\
      & \lambda_* &&\text{if} \ \ p_f=p_*, 
      \\
      & \frac{a}{2} &&\text{if} \ \ p_*<p_f<p_\mathrm{S}.
    \end{aligned}
  \right.
  \end{equation}
  We comment on the behavior of $\eta$ as $\rho\to\infty$ in the following two cases$:$ 
  \begin{center}
    \begin{tabular}{>{\centering}p{0.45\linewidth}>{\centering}p{0.45\linewidth}}
      $\displaystyle \mathrm{Case~(A)} \quad \int_{\rho_0}^{\infty}e^{\Lambda \tau}|I(\tau)| \, d\tau <\infty;$
      & 
      $\displaystyle \mathrm{Case~(B)} \quad \int_{\rho_0}^{\infty}e^{\Lambda \tau}|I(\tau)| \, d\tau =\infty.$
    \end{tabular}  
  \end{center}
  In case~$\mathrm{(A)}$, it follows that 
  \begin{equation}\label{eq:3.3}
    \int_{\rho_0}^{\rho}Q(\rho,\tau) |I(\tau)|\,d\tau
    \lesssim 
    Q(\rho,\rho_0)
    \int_{\rho_0}^{\rho}e^{\Lambda (\tau-\rho_0)}|I(\tau)|\, d\tau
    \lesssim
    e^{-{\Lambda}\rho_0} Q(\rho,\rho_0)
  \end{equation}
  for all $\rho\ge \rho_0$. 
  Thus, the size of the solution $\eta$ is controlled by a super-kernel $Q$. 
  In particular, when $p_\mathrm{c}<p_f\le p_*$, by \eqref{eq:2.15} we may assume that 
  the behavior of $\eta$ as $\rho \to \infty$
  is mainly determined by the fundamental function $\Phi_1$ or $\Phi_2$ of \eqref{eq:2.13}. 

  On the other hand, 
  in case~$\mathrm{(B)}$, we see that 
  \begin{equation}\label{eq:3.4}
    \lim_{\rho \to \infty}
    \frac{\displaystyle\int_{\rho_0}^{\rho}Q(\rho,\tau) |I(\tau)|\, d\tau}{Q(\rho,\rho_0)}
    =\infty, 
  \end{equation}
  which clearly holds if $p_f\neq p_*$ 
  and follows from 
  \[
    \frac{\displaystyle\int_{\rho_0}^{\frac{\rho}{2}}
    (1+(\rho-\tau))e^{-\lambda_*(\rho-\tau)} 
    |I(\tau)|\, d\tau}{(1+(\rho-\rho_0))e^{-\lambda_*(\rho-\rho_0)}}
    \ge
    \frac{
      \left(1+\frac{\rho}{2}\right)e^{-\lambda_*\rho}
      \displaystyle\int_{\rho_0}^{\frac{\rho}{2}}e^{\lambda_*\tau} |I(\tau)|\,d\tau
    }{(1+(\rho-\rho_0))e^{-\lambda_*(\rho-\rho_0)}}
    \to \infty
  \]
  as $\rho \to \infty$ if $p_f=p_*$. 
  Therefore, the effect of a super-kernel $Q$ becomes smaller in the behavior of $\eta$ as $\rho\to\infty$. 
  In particular, when $p_\mathrm{c}<p_f\le p_*$, 
  the principal part of the behavior of $\eta$ can be considered to be given by 
  $\displaystyle\int_{\rho_0}^{\rho} Q(\rho,\tau) |I(\tau)|\, d\tau$, 
  which corresponds to the inhomogeneous term $\displaystyle\int_{\rho_0}^{\rho}K(\rho,\tau) I(\tau)\, d\tau$ if $I$ is positive. 

  In both cases, 
  $\eta(\rho)\to 0$ as $\eta\to \infty$ by Lemma~$\mathrm{\ref{Lemma:2.2}}$, 
  so the singularity of $u$ defined by \eqref{eq:2.5} is preserved, 
  and it can be seen that $u$ is a singular solution of \eqref{eq:1.1}. 
\end{remark}

The rest of this section is devoted to the proof of Theorem~\ref{Theorem:3.1}. 

\begin{proof}[Proof of Theorem~$\mathrm{\ref{Theorem:3.1}}$]
  Let $\rho_0$ and $\epsilon$ be positive constants to be chosen later and $\alpha$, $\beta>0$ satisfy $\alpha+\beta<\epsilon$. 
  For positive constants $C_1$ and $C_2$ which will be chosen later, set 
  \[
  \Phi(\rho)
  :=
  C_1 \Phi_1(\rho)+C_2 \Phi_2(\rho),
  \]
  where $\Phi_1$ and $\Phi_2$ are fundamental functions of equation \eqref{eq:2.13}.
  We denote the right hand side of \eqref{eq:2.15} by $T[\eta](\rho)$, that is,
  \[
    T[\eta](\rho):= \Phi(\rho)-\int_{\rho_0}^{\rho}K(\rho,\tau)  
    (I(\tau)+L_1(\tau)\eta(\tau)+L_2(\tau)\eta'(\tau)+N[\eta](\tau))\, d\tau.
  \]
  Thanks to \eqref{eq:3.1}, 
  we have
  \begin{equation}\label{eq:3.5}
    \begin{aligned}
      &
      |T[\eta](\rho)|
      +
      \left|\frac{d}{d\rho}T[\eta](\rho)\right|
      \\
      &
      \lesssim 
      |\Phi(\rho)|+|\Phi'(\rho)|
      +\int_{\rho_0}^{\rho} Q(\rho,\tau)
      (|I(\tau)|+|L_1(\tau)||\eta(\tau)|+|L_2(\tau)||\eta'(\tau)|+|N[\eta](\tau)|)\, d\tau.
    \end{aligned}
  \end{equation}
  We now define the class of solutions as follows:
  \begin{equation}\label{eq:3.6}
    \begin{gathered}
      \|\eta\|_{\rho_0}
      :=
      \sup_{\rho\ge \rho_0} \frac{|\eta(\rho)|+|\eta'(\rho)|}
      {
      \delta Q(\rho,\rho_0)+\displaystyle\int_{\rho_0}^{\rho}Q(\rho,\tau)|I(\tau)|\, d\tau
      },
      \\
      X_{\rho_0}
      :=\left\{\eta\in C^1((\rho_0,\infty))\colon\|\eta\|_{\rho_0}\le 2,\ \eta(\rho_0)=\alpha,\ \eta'(\rho_0)=\beta \right\},
    \end{gathered}
  \end{equation}
  where $\delta>0$ is a constant to be chosen later.

  To show that $T$ is a map from $X_{\rho_0}$ to itself, 
  we start by estimating the fundamental function term $|\Phi(\rho)|+|\Phi'(\rho)|$.
  If $T[\eta]$ satisfies the initial condition $T[\eta](\rho_0)=\alpha$ and $(T[\eta])'(\rho_0)=\beta$, 
  then $C_1$ and $C_2$ must satisfy
  \begin{equation*} 
    \alpha
    =C_1\Phi_1(\rho_0)+C_2\Phi_2(\rho_0),
    \qquad
    \beta
    =C_1\Phi_1'(\rho_0)+C_2\Phi_2'(\rho_0),
  \end{equation*}
  and hence
  \begin{equation}\label{eq:3.7}
    C_1
    =
    \frac{\Phi_2'(\rho_0)\alpha-\Phi_2(\rho_0)\beta}{W(\rho_0)},
    \qquad 
    C_2
    =
    \frac{-\Phi_1'(\rho_0)\alpha+\Phi_1(\rho_0)\beta}{W(\rho_0)}.
  \end{equation}
  Vise versa, if $C_1$ and $C_2$ satisfy \eqref{eq:3.7}, then $T[\eta]$ satisfies the initial condition, i.e., 
  $T[\eta](\rho_0)=\alpha$ and $(T[\eta])'(\rho_0)=\beta$. 
  Therefore, we take $C_1$ and $C_2$ as in \eqref{eq:3.7}.
  It follows from \eqref{eq:3.7} that 
  \[
  \Phi(\rho)
  =
  C_1\Phi_1(\rho)+C_2\Phi_2(\rho)
  =
  \frac{\Phi_2'(\rho_0)\Phi_1(\rho)-\Phi_1'(\rho_0)\Phi_2(\rho)}{W(\rho_0)}\alpha+K(\rho,\rho_0)\beta.
  \]
  By simple computations, we have
  \begin{equation}\label{eq:3.8}
    \begin{aligned}
      \left|
      \frac{\Phi_2'(\rho_0)\Phi_1(\rho)-\Phi_1'(\rho_0)\Phi_2(\rho)}{W(\rho_0)}
      \right|
      +
      \left|
      \frac{d}{d\rho}
      \left(
      \frac{\Phi_2'(\rho_0)\Phi_1(\rho)-\Phi_1'(\rho_0)\Phi_2(\rho)}{W(\rho_0)}
      \right)
      \right|
      \lesssim
      Q(\rho,\rho_0).
    \end{aligned}
  \end{equation}
  In fact,
  if $p_\mathrm{c}<p_f<p_*$, then
  \[
    \begin{gathered}
      \frac{\Phi_2'(\rho_0)\Phi_1(\rho)-\Phi_1'(\rho_0)\Phi_2(\rho)}{W(\rho_0)}
      =
      \frac{\lambda_2}{\lambda_2-\lambda_1}e^{-\lambda_1(\rho-\rho_0)}
      -
      \frac{\lambda_1}{\lambda_2-\lambda_1}e^{-\lambda_2(\rho-\rho_0)},
      \\
      \frac{d}{d\rho}
      \left(\frac{\Phi_2'(\rho_0)\Phi_1(\rho)-\Phi_1'(\rho_0)\Phi_2(\rho)}{W(\rho_0)}\right)
      =
      -\frac{\lambda_1\lambda_2}{\lambda_2-\lambda_1}
      \left(
      e^{-\lambda_1(\rho-\rho_0)}
      -
      e^{-\lambda_2(\rho-\rho_0)}
      \right),
    \end{gathered}
  \]
  \item
  if $p_f=p_*$, then
  \[
    \begin{gathered}
      \frac{\Phi_2'(\rho_0)\Phi_1(\rho)-\Phi_1'(\rho_0)\Phi_2(\rho)}{W(\rho_0)}
      =
      (1-\lambda_*(\rho-\rho_0))
      e^{-\lambda_*(\rho-\rho_0)},
      \\
      \frac{d}{d\rho}\left(\frac{\Phi_2'(\rho_0)\Phi_1(\rho)-\Phi_1'(\rho_0)\Phi_2(\rho)}{W(\rho_0)}\right)
      =
      -\lambda_*^2(\rho-\rho_0)
      e^{-\lambda_*(\rho-\rho_0)}, 
    \end{gathered}
  \]
  and if 
  $p_*<p_f<p_\mathrm{S}$, then
  \[
    \begin{gathered}
      \frac{\Phi_2'(\rho_0)\Phi_1(\rho)-\Phi_1'(\rho_0)\Phi_2(\rho)}{W(\rho_0)}
      =
      \frac{1}{k}
      e^{-\frac{a}{2}(\rho-\rho_0)}
      \left(
      \frac{a}{2}\sin k(\rho-\rho_0)+k\cos k(\rho-\rho_0)
      \right),
      \\
      \frac{d}{d\rho}\left(\frac{\Phi_2'(\rho_0)\Phi_1(\rho)-\Phi_1'(\rho_0)\Phi_2(\rho)}{W(\rho_0)}\right)
      =
      -\frac{1}{k}e^{-\frac{a}{2}(\rho-\rho_0)}
      \left(
      \frac{a^2}{4}+k^2
      \right)\sin k(\rho-\rho_0).
    \end{gathered}
  \]
  Therefore, \eqref{eq:3.8} holds for all cases. 
  Combining \eqref{eq:3.8} with \eqref{eq:3.1}, we have
  \begin{equation}\label{eq:3.9}
    |\Phi(\rho)|+|\Phi'(\rho)|
    \lesssim 
    (\alpha+\beta)Q(\rho,\rho_0)
  \end{equation}
  for all $\rho\ge\rho_0$. 

  It remains to estimate 
  the linear terms $L_1\eta,L_2\eta'$ and the nonlinear term $N[\eta]$ for $\eta \in X_{\rho_0}$. 
  Let $\eta \in X_{\rho_0}$. 
  By \eqref{eq:3.1}, a linear term including $L_1\eta$ satisfies
  \begin{equation}\label{eq:3.10}
    \begin{aligned}
      &
      \left| \int_{\rho_0}^{\rho} K(\rho,\tau)L_1(\tau)\eta(\tau)\, d\tau \right| 
      \\
      &
      \lesssim 
      \int_{\rho_0}^{\rho} Q(\rho,\tau)|L_1(\tau)|
      \delta Q(\tau,\rho_0)
      d\tau
      +
      \int_{\rho_0}^{\rho} Q(\rho,\tau)|L_1(\tau)|
      \left(
      \int_{\rho_0}^{\tau}Q(\tau,t)|I(t)|dt
      \right)
      d\tau
    \end{aligned}
  \end{equation}
  for $\eta\in X_{\rho_0}$. 
  By \eqref{eq:2.8}, \eqref{eq:2.10} and Lemmas~\ref{Lemma:2.1} and \ref{Lemma:2.4}, 
  the first term on the right hand side of \eqref{eq:3.10} satisfies
  \begin{equation}\label{eq:3.11}
    \begin{aligned}
      \int_{\rho_0}^{\rho} Q(\rho,\tau)|L_1(\tau)|
      \delta Q(\tau,\rho_0)\,
      d\tau
      &
      \lesssim
      \delta 
      \left(
      \sup_{\tau \ge \rho_0}Q(\tau,\rho_0)
      \right)
      \int_{\rho_0}^{\rho} Q(\rho,\tau)|L_1(\tau)|\, 
      d\tau
      \\
      &
      \lesssim
      \delta 
      \left(
      \int_{\rho_0}^{\rho} Q(\rho,\tau)|I(\tau)|\,
      d\tau
      +
      \epsilon_{\rho_0} Q(\rho,\rho_0)
      \right), 
    \end{aligned}
  \end{equation}
  where $\epsilon_{\rho_0}$ is the constant defined by 
  \begin{equation}
    \notag 
    \epsilon_{\rho_0} := 
    \left|
      \frac{f(\phi(\rho_0))F(\phi(\rho_0))}{\phi(\rho_0)} 
      - \frac{1}{p_f-1} 
    \right|. 
  \end{equation}
  Meanwhile, similarly to \eqref{eq:3.11}, 
  the second term on the right hand side of \eqref{eq:3.10} satisfies
  \[
    \begin{aligned}
      &
      \int_{\rho_0}^{\rho} Q(\rho,\tau)|L_1(\tau)|
      \left(
        \int_{\rho_0}^{\tau}Q(\tau,t)|I(t)|\, dt
      \right)
      d\tau
      \\
      &
      \lesssim
      \left(
        \sup_{\tau\ge \rho_0}
        \int_{\rho_0}^{\tau}Q(\tau,t)|I(t)|\, dt
      \right)
        \int_{\rho_0}^{\rho}Q(\rho,\tau)|L_1(\tau)|\, d\tau
      \\
      &
      \lesssim
      \left(
        \sup_{\tau\ge \rho_0}
        \int_{\rho_0}^{\tau}Q(\tau,t)|I(t)|\, dt
      \right)
      \left(
        \int_{\rho_0}^{\rho}Q(\rho,\tau)|I(\tau)|
        \, d\tau
        +
        \epsilon_{\rho_0} Q(\rho,\rho_0)
      \right).
    \end{aligned}
  \]
  Taking sufficiently small $\delta$ 
  and
  then taking sufficiently large $\rho_0$, we see from
  Lemma~\ref{Lemma:2.2} (ii) and (iv) that 
  \begin{equation}\label{eq:3.12}
    \begin{aligned}
      &
      \int_{\rho_0}^{\rho} Q(\rho,\tau)
      |L_1(\tau)||\eta(\tau)|\, d\tau
      \le
      \frac{1}{4}
      \left(
        \int_{\rho_0}^{\rho}Q(\rho,\tau)|I(\tau)|
        d\tau
        +
        \delta Q(\rho,\rho_0)
      \right)
    \end{aligned}
  \end{equation}
  for all $\rho\ge \rho_0$. 
  By the same argument as above, it holds
  \begin{equation}\label{eq:3.13}
    \begin{aligned}
      &
      \int_{\rho_0}^{\rho}Q(\rho,\tau)
      |L_2(\tau)||\eta'(\tau)|\, d\tau
      \le
      \frac{1}{4}
      \left(
        \int_{\rho_0}^{\rho} Q(\rho,\tau)|I(\tau)|
        d\tau
        +
        \delta Q(\rho,\rho_0)
      \right)
    \end{aligned}
  \end{equation}
  for all $\rho\ge \rho_0$. 

  It remains to estimate the nonlinear term. 
  We divide the proof into the two cases: 
  $p_f=p_*$ and $p_f\neq p_*$. 

  \medskip 
  \noindent
  \underline{Case $p_f=p_*$}\quad 
  We first consider the case $p_f=p_*$.
  Lemma~\ref{Lemma:2.3} with $\eta_1=\eta$ and $\eta_2=0$ and \eqref{eq:3.1} 
  yield 
  \begin{equation}\label{eq:3.14}
    \begin{aligned}
      &
      \left| \int_{\rho_0}^{\rho}K(\rho,\tau)N[\eta](\tau)\, d\tau \right| 
      \\
      &
      \lesssim
      \int_{\rho_0}^{\rho}(1+(\rho-\tau))e^{-\lambda_*(\rho-\tau)}
      |\eta(\tau)|^2\, 
      d\tau
      \\
      &
      \lesssim
      \delta^2
      \int_{\rho_0}^{\rho}( 1+(\rho-\tau) )e^{-\lambda_*(\rho-\tau)}
      (1+(\tau-\rho_0))^2 e^{-2\lambda_*(\tau-\rho_0)}\, 
      d\tau
      \\
      &
      \qquad
      +
      \int_{\rho_0}^{\rho}(1+(\rho-\tau))e^{-\lambda_*(\rho-\tau)}
      \left(
        \int_{\rho_0}^{\tau}(1+(\tau-t))e^{-\lambda_*(\tau-t)}|I(t)|\, dt
      \right)^2
      d\tau
      \\
      &
      =:N_1+N_2.
    \end{aligned}
  \end{equation}
  $N_1$ can be estimated as follows: 
  \begin{equation}\label{eq:3.15}
    \begin{aligned}
      N_1
      &
      \lesssim
      \delta^2
      (1+(\rho-\rho_0))e^{-\lambda_*(\rho-\rho_0)}
      \int_{\rho_0}^{\infty}
      (1+(\tau-\rho_0))^2e^{-\lambda_*(\tau-\rho_0)}\, 
      d\tau
      \\
      &
      \lesssim
      \delta^2
      (1+(\rho-\rho_0))e^{-\lambda_*(\rho-\rho_0)}.
    \end{aligned}
  \end{equation}
  Meanwhile, for $N_2$, by integration by parts we have 
  \[
    \begin{aligned}
      N_2
      &
      = 
      e^{-\lambda_*(\rho-\rho_0)}
      \int_{\rho_0}^{\rho}
      (1+(\rho-\tau))
      e^{-\lambda_*(\tau-\rho_0)}
      \left(
        \int_{\rho_0}^{\tau}(1+(\tau-t))e^{\lambda_*(t-\rho_0)}|I(t)|\, dt
      \right)^2
      d\tau
      \\
      &
      =
      e^{-\lambda_*(\rho-\rho_0)}
      \left[
        -\frac{e^{-\lambda_*(\tau-\rho_0)}}{\lambda_*}
        (1+(\rho-\tau))
        \left(
          \int_{\rho_0}^{\tau}(1+(\tau-t))e^{\lambda_*(t-\rho_0)}|I(t)|\, dt
        \right)^2
      \right]_{\tau=\rho_0}^{\tau=\rho}
      \\
      &
      \quad
      +
      \frac{e^{-\lambda_*(\rho-\rho_0)}}{\lambda_*}
      \int_{\rho_0}^{\rho}
      e^{-\lambda_*(\tau-\rho_0)}
      (-1)
      \left(
        \int_{\rho_0}^{\tau}(1+(\tau-t))e^{\lambda_*(t-\rho_0)}|I(t)|\, dt
      \right)^2
      d\tau
      \\
      &
      \quad
      +
      \frac{2e^{-\lambda_*(\rho-\rho_0)}}{\lambda_*}
      \int_{\rho_0}^{\rho}
      e^{-\lambda_*(\tau-\rho_0)}
      (1+(\rho-\tau))
      \left(
        e^{\lambda_*(\tau-\rho_0)}|I(\tau)|
        +
        \int_{\rho_0}^{\tau}e^{\lambda_*(t-\rho_0)}|I(t)|\, dt
      \right)
      \\
      &
      \hspace{20.5em}
      \times
      \left(\int_{\rho_0}^{\tau}(1+(\tau-t))e^{\lambda_*(t-\rho_0)}|I(t)|\, dt\right)
      d\tau.
    \end{aligned}
  \]
  One can neglect the first and the second term since they are negative, hence we only 
  estimate the third term. 
  Setting 
  \[
    \begin{aligned}
      N_{2,1}
      &
      :=
      \int_{\rho_0}^{\rho}
      e^{-\lambda_*(\tau-\rho_0)}
      (1+(\rho-\tau))
      e^{\lambda_*(\tau-\rho_0)}
      |I(\tau)|
      \left(
        \int_{\rho_0}^{\tau}(1+(\tau-t))e^{\lambda_*(t-\rho_0)}|I(t)|\, dt
      \right)
      d\tau,
      \\
      N_{2,2}
      &
      :=
      \int_{\rho_0}^{\rho}
      e^{-\lambda_*(\tau-\rho_0)}
      (1+(\rho-\tau))
      \left(
        \int_{\rho_0}^{\tau}e^{\lambda_*(t-\rho_0)}|I(t)|\, dt
        \int_{\rho_0}^{\tau}(1+(\tau-t))e^{\lambda_*(t-\rho_0)}|I(t)|\, dt
      \right)
      d\tau, 
    \end{aligned}
  \]
  we have \[ N_2\le \frac{2e^{-\lambda_*(\rho-\rho_0)}}{\lambda_*}(N_{2,1} + N_{2,2}).\] 
  One can easily estimate $N_{2,1}$ as follows: 
  \begin{equation}\label{eq:3.16}
  \begin{aligned}
  &
  \frac{2e^{-\lambda_*(\rho-\rho_0)}}{\lambda_*}N_{2,1} 
  \\
  &
  \le 
  \frac{2}{\lambda_*}
  \left(
    \sup_{\tau\ge \rho_0}
    \int_{\rho_0}^{\tau}(1+(\tau-t))e^{-\lambda_*(\tau-t)}|I(t)|\, dt
  \right)
  \int_{\rho_0}^{\rho}
  (1+(\rho-\tau))
  e^{-\lambda_*(\rho-\tau)}|I(\tau)|\, d\tau.
  \end{aligned}
  \end{equation}
  An estimate for $N_{2,2}$ is more involved. 
  Applying integration by parts again and neglecting negative parts similar to the previous argument, we have
  \[
    \begin{aligned}
      N_{2,2}
      &
      \le
      \frac{1}{\lambda_*} 
      \int_{\rho_0}^{\rho}
      e^{-\lambda_*(\tau-\rho_0)}
      (1+(\rho-\tau))
      e^{\lambda_*(\tau-\rho_0)}I(\tau)
      \left(
        \int_{\rho_0}^{\tau}(1+(\tau-t))e^{\lambda_*(t-\rho_0)}I(t)\, dt
      \right)
      d\tau
      \\
      &
      \quad 
      +
      \frac{1}{\lambda_*} 
      \int_{\rho_0}^{\rho}
      e^{-\lambda_*(\tau-\rho_0)}
      (1+(\rho-\tau))
      \left(
        \int_{\rho_0}^{\tau}e^{\lambda_*(t-\rho_0)}I(t)\, dt
      \right)
      e^{\lambda_*(\tau-\rho_0)}I(\tau)\, 
      d\tau 
      \\
      &
      \quad 
      +
      \frac{1}{\lambda_*} 
      \int_{\rho_0}^{\rho}
      e^{-\lambda_*(\tau-\rho_0)}
      (1+(\rho-\tau))
      \left(
        \int_{\rho_0}^{\tau}e^{\lambda_*(t-\rho_0)}I(t)\, dt
      \right)^2
      d\tau.
    \end{aligned}
  \]
  Therefore, similarly to \eqref{eq:3.16}, we obtain
  \begin{equation}\label{eq:3.17}
    \begin{aligned}
      &
      \frac{2e^{-\lambda_*(\rho-\rho_0)}}{\lambda_*}N_{2,2}
      \\
      &
      \le 
      \frac{2}{\lambda_*^2}
      \left(
      \sup_{\tau\ge \rho_0}
      \int_{\rho_0}^{\tau}(1+(\tau-t))e^{-\lambda_*(\tau-t)}I(t)\, dt
      +
      \sup_{\tau\ge \rho_0}
      \int_{\rho_0}^{\tau}e^{-\lambda_*(\tau-t)}I(t)\, dt
      \right)
      \\
      &
      \hspace{16.5em}
      \times
      \int_{\rho_0}^{\rho}
      (1+(\rho-\tau))
      e^{-\lambda_*(\rho-\tau)}I(t)\, d\tau
      \\
      &
      \qquad
      +
      \frac{2e^{-\lambda_*(\rho-\rho_0)}}{\lambda_*^2}
      \int_{\rho_0}^{\rho}
      e^{-\lambda_*(\tau-\rho_0)}
      (1+(\rho-\tau))
      \left(
        \int_{\rho_0}^{\tau}e^{\lambda_*(t-\rho_0)}I(t)\, dt
      \right)^2
      d\tau.
    \end{aligned}
  \end{equation}
  Applying integration by parts again to the last term on the right hand side of \eqref{eq:3.17}, we obtain
  \[
    \begin{aligned}
      &
      \frac{2e^{-\lambda_*(\rho-\rho_0)}}{\lambda_*^2}
      \int_{\rho_0}^{\rho}
      e^{-\lambda_*(\tau-\rho_0)}
      (1+(\rho-\tau))
      \left(
        \int_{\rho_0}^{\tau}e^{\lambda_*(t-\rho_0)}|I(t)|\, dt
      \right)^2
      d\tau
      \\
      &
      \le
      \frac{4e^{-\lambda_*(\rho-\rho_0)}}{\lambda_*^3}
      \int_{\rho_0}^{\rho}
      e^{-\lambda_*(\tau-\rho_0)}
      (1+(\rho-\tau))
      e^{\lambda_*(\tau-\rho_0)}|I(\tau)|
      \left(
        \int_{\rho_0}^{\tau}e^{\lambda_*(t-\rho_0)}|I(t)|\, dt
      \right)
      d\tau
      \\
      &
      \le
      \frac{4}{\lambda_*^3}
      \left(\sup_{\tau\ge \rho_0}\int_{\rho_0}^{\tau}e^{-\lambda_*(\tau-t)}|I(t)|\, dt\right)
      \int_{\rho_0}^{\rho}
      (1+(\rho-\tau))
      e^{-\lambda_*(\rho-\tau)}|I(\tau)|\, d\tau.
    \end{aligned}
  \]
  All together, for sufficiently small $\delta>0$ and sufficiently large $\rho_0$, 
  we obtain 
  \begin{equation}\label{eq:3.18}
    \left| \int_{\rho_0}^{\rho}K(\rho,\tau)N[\eta](\tau)\, d\tau \right| 
    \lesssim 
    N_1+N_2
    \le 
    \frac{1}{4}
    \left(
      \int_{\rho_0}^{\rho} Q(\rho,\tau)|I(\tau)|\, 
      d\tau
      +
      \delta Q(\rho,\rho_0)
    \right)
  \end{equation}
  for all $\rho\ge \rho_0$ with the help of Lemma~\ref{Lemma:2.2}~(iv). 

  \medskip 

  \noindent 
  \underline{Case $p_f\neq p_*$}\quad 
  We next consider the case $p_f\neq p_*$. 
  In this case, it holds 
  $
  Q(\rho,\tau)=e^{-\Lambda(\rho-\tau)}.
  $
  Here $\Lambda$ is the constant given by \eqref{eq:3.2}. 
  Similarly to the above argument for the case $p_f=p^*$, 
  we have
  \begin{equation}\label{eq:3.19}
    \begin{aligned}
      &
      \left| \int_{\rho_0}^{\rho}K(\rho,\tau)N[\eta](\tau)\, d\tau \right| 
      \lesssim
      \int_{\rho_0}^{\rho}e^{-\Lambda(\rho-\tau)}
      |\eta(\tau)|^2\, d\tau
      \\
      &
      \quad 
      \lesssim
      \delta^2
      \int_{\rho_0}^{\rho}e^{-\Lambda(\rho-\tau)}
      e^{-2\Lambda(\tau-\rho_0)}
      d\tau
      +
      \int_{\rho_0}^{\rho}
      e^{-\Lambda(\rho-\tau)}
      \left(
        \int_{\rho_0}^{\tau}e^{-\Lambda(\tau-t)}|I(t)|\, dt
      \right)^2
      d\tau
      \\
      &
      \quad 
      \lesssim
      \delta^2 e^{-\Lambda (\rho-\rho_0)}
      +
      e^{-\Lambda(\rho-\rho_0)}
      \int_{\rho_0}^{\rho}
      e^{-\Lambda(\tau-\rho_0)}
      \left(
        \int_{\rho_0}^{\tau}e^{\Lambda(t-\rho_0)}|I(t)|\, dt
      \right)^2
      d\tau
    \end{aligned}
  \end{equation}
  for all $\rho\ge \rho_0$. 
  Applying integration by parts and neglecting the negative part yields 
  \[
    \begin{aligned}
      &
      e^{-\Lambda(\rho-\rho_0)}
      \int_{\rho_0}^{\rho}
      e^{-\Lambda(\tau-\rho_0)}
      \left(\int_{\rho_0}^{\tau}e^{\Lambda(t-\rho_0)}|I(t)|\, dt\right)^2
      d\tau
      \\
      &
      \le
      \frac{2}{\Lambda}
      e^{-\Lambda(\rho-\rho_0)}
      \int_{\rho_0}^{\rho}
      e^{-\Lambda(\tau-\rho_0)}
      e^{\Lambda(\tau-\rho_0)}|I(\tau)|
      \left(
        \int_{\rho_0}^{\tau}e^{\Lambda(t-\rho_0)}|I(t)|\, dt
      \right)
      d\tau
      \\
      &
      \le
      \frac{2}{\Lambda}
      \left(\sup_{\tau \ge \rho_0}\int_{\rho_0}^{\tau}e^{-\Lambda(\tau-t)}|I(t)|\, dt\right)
      \int_{\rho_0}^{\rho}
      e^{-\Lambda(\rho-\tau)}
      |I(\tau)|\, d\tau.
    \end{aligned}
  \]
  Therefore, 
  combining all of the above and 
  taking sufficiently small $\delta>0$ and sufficiently large $\rho_0$, we obtain
  \begin{equation}\label{eq:3.20}
    \begin{aligned}
      \left| \int_{\rho_0}^{\rho}K(\rho,\tau)N[\eta](\tau)\, d\tau \right| 
      \le
      \frac{1}{4}
      \left(
        \int_{\rho_0}^{\rho} Q(\rho,\tau)|I(\tau)|\, d\tau
        +
        \delta Q(\rho,\rho_0)
      \right)
    \end{aligned}
  \end{equation}
  for all $\rho\ge \rho_0$ with the help of Lemma~\ref{Lemma:2.2}~(iv). 

  \medskip 
  We are ready to conclude the proof. 
  We first take a sufficiently small $\delta>0$, then take sufficiently large $\rho_0>0$, 
  thus we see from \eqref{eq:3.12}, \eqref{eq:3.13}, \eqref{eq:3.18} and \eqref{eq:3.20} that 
  for $|\alpha|+|\beta|<\frac{\delta}{4}$, 
  \[
    |T[\eta](\rho)|
    +
    \left|\frac{d}{d\rho}T[\eta](\rho)\right|
    \le
    \delta Q(\rho,\rho_0)+\frac{7}{4}\int_{\rho_0}^{\rho}Q(\rho,\tau)|I(\tau)|\, d\tau,
  \]
  which implies that 
  $T$ is a map from $X_{\rho_0}$ to itself. 
  Similarly, 
  if $\delta>0$ is sufficiently small and $\rho_0>0$ is sufficiently large, 
  then one can prove that
  \[
  \|T[\eta_1]-T[\eta_2]\|_{\rho_0}\le \frac{1}{2}\|\eta_1-\eta_2\|_{\rho_0}
  \]
  for all $\eta_1,\eta_2\in X_{\rho_0}$. 
  Therefore, the existence of a solution of \eqref{eq:2.15} can be shown using 
  the fixed point theorem, 
  and the proof of Theorem~\ref{Theorem:3.1} is complete. 
\end{proof}

\section{Examples}
\label{section:4}
In this section, we show some examples to which our theorem applies. 
Since the behavior of $\theta(r)$ is converted to the behavior of $\eta(\rho)$ by $\theta(r)=\eta(\rho)$ and $\rho=\log (1/r)$ (see \eqref{eq:2.7}), 
we only examine the latter.
Note that Corollary~\ref{Corollary:1.1} is a direct consequence of Example~\ref{Example:4.1}. 
\begin{example}
  \label{Example:4.1}
  Let $f(s)=s^p+s^r$ with $p_\mathrm{c}<p<p_\mathrm{S}$ and $1<r<p$. 
  Then $p_f=p$ and 
  \begin{align}
    \label{eq:4.1}
    f'(s)
    &
    =ps^{p-1}+rs^{r-1}, 
    \\
    \label{eq:4.2}
    F(s)
    &
    =\sum_{k=0}^{\infty}\frac{(-1)^k }{p+k(p-r)-1}s^{1-p-k(p-r)},
    \\
    \label{eq:4.3}
    f'(s)F(s)
    &
    =
    \frac{p}{p-1}+\sum_{k=1}^{\infty}
    \frac{(p-r)(1-k(p-r))}{(p+k(p-r)-1)(p+(k-1)(p-r)-1)}
    (-1)^{k-1}s^{-k(p-r)},
    \\
    \label{eq:4.4}
    \frac{f(s)F(s)}{s}
    &
    =
    \frac{1}{p-1}+\sum_{k=1}^{\infty}
    \frac{p-r}{(p+k(p-r)-1)(p+(k-1)(p-r)-1)}(-1)^{k-1}s^{-k(p-r)},
    \end{align}
  for all $s>1$, and 
  \begin{equation}
    \label{eq:4.5}
    I(\rho)\simeq \frac{1}{\phi(\rho)^{p-r}}\simeq e^{-\frac{2(p-r)}{p-1}\rho}
  \end{equation}
  for all sufficiently large $\rho>0$. 
  Let $r^*$ be the exponent defined by \eqref{eq:1.16}. 
  Then $r^*$ satisfies $0<r^*<p$ and 
  \[
    \frac{2(p-r^*)}{p-1}=\Lambda,
  \]
  yielding 
  \begin{equation}
    \notag 
    \int_{\rho_0}^\infty e^{\Lambda \tau}|I(\tau)| \, d\tau
    \left\{
      \begin{aligned}
        & <\infty &&\text{if} \ \ 0<r<r_*,
        \\
        & =\infty &&\text{if} \ \ r_*\le r<p.
      \end{aligned}
    \right.
  \end{equation}
  Here $\Lambda$ is defined by \eqref{eq:3.2}.
  Therefore, keeping Remark~$\mathrm{\ref{Remark:3.1}}$ in mind and applying Theorem~$\mathrm{\ref{Theorem:3.1}}$, 
  we obtain a constant multiple of the functions in Table~$\mathrm{\ref{table:2}}$ 
  as an upper bound for $|\eta(\rho)|+|\eta'(\rho)|$ as $\rho\to\infty$. 
  \begin{table}[h]
    \begin{center}
      {\renewcommand\arraystretch{1.8}
      \begin{tabular}{|c||c|c|c|}
        \hline 
        $|\eta(\rho)|+|\eta'(\rho)|$
        & $p_\mathrm{c}<p<p_*$ 
        & $p=p_*$ 
        & $p_*<p<p_\mathrm{S}$  
        \\[1mm]
        \hline\hline 
        $\displaystyle 0< r<r^*$ 
        & $\displaystyle e^{-\lambda_1 \rho}$
        & $\displaystyle \rho e^{-\lambda_*\rho}$
        &
        $\displaystyle e^{-\frac{a}{2}\rho}$
        \\[1mm]
        \hline
        $\displaystyle r=r^*$ 
        & $\displaystyle \rho e^{-\lambda_1 \rho}$
        & $\displaystyle \rho^2 e^{-\lambda_*\rho}$
        &
        $\displaystyle \rho e^{-\frac{a}{2}\rho}$
        \\[1mm]
        \hline
        \rule[0cm]{-1.5mm}{8mm}
        \raisebox{3.5pt}{$\displaystyle r^*<r<p$}
        & 
        \multicolumn{3}{c|}{
          \raisebox{3.5pt}{$\displaystyle  e^{-\frac{2(p-r)}{{p-1}}\rho}$}
        }
        \\
        \hline 
      \end{tabular}
      \caption{Upper estimates of decay rates of $|\eta(\rho)|+|\eta'(\rho)|$}\label{table:2}
      }
    \end{center}
  \end{table}
\end{example}

\begin{proof}
  Equality \eqref{eq:4.1} is obvious. 
  Since equalities \eqref{eq:4.3} and \eqref{eq:4.4} are easy to follow from \eqref{eq:4.1} and \eqref{eq:4.2}, 
  it remains to show equality \eqref{eq:4.2}.
  By Taylor's expansion of $1/f(s)$,
  we have
  \[
    \begin{aligned}
      F(s)
      =
      \int_s^{\infty}\frac{1}{u^p\left(1+\frac{1}{u^{p-r}}\right)}\, du
      =
      \sum_{k=0}^{\infty}\int_s^{\infty}
      (-1)^k \frac{1}{u^{p+k(p-r)}}\, du,
    \end{aligned}
  \]
  yielding \eqref{eq:4.2}.
  Estimate~\eqref{eq:4.5} can be seen from \eqref{eq:2.4}, \eqref{eq:2.7}, \eqref{eq:2.8}, \eqref{eq:4.2}, \eqref{eq:4.3} and \eqref{eq:4.4}. 
  Table~2
  can be obtained by Theorem~$\mathrm{\ref{Theorem:3.1}}$ and Remark~$\mathrm{\ref{Remark:3.1}}$ with simple calculations. 
\end{proof}

\begin{example}
  \label{Example:4.2}
  Let $f(s)=s^p(\log s)^r$ for $s>2$ with $p>1$ and $r\in \mathbb{R}$. 
  Then $p_f=p$ and 
  \begin{align}
    \label{eq:4.6}
    f'(s)
    &
    =ps^{p-1}(\log s)^r+rs^{p-1}(\log s)^{r-1}, 
    \\
    \label{eq:4.7}
    F(s)
    &
    =
    \frac{1}{p-1}\frac{1}{s^{p-1}(\log s)^r}-\frac{r}{(p-1)^2}\frac{1}{s^{p-1}(\log s)^{r+1}}
    +O\left(\frac{1}{s^{p-1}(\log s)^{r+2}}\right),
    \\
    \label{eq:4.8}
    f'(s)F(s)
    &
    =
    \frac{p}{p-1}-\frac{r}{(p-1)^2}\frac{1}{\log s}
    +O\left(\frac{1}{(\log s)^{2}}\right),
    \\
    \label{eq:4.9}
    \frac{f(s)F(s)}{s}
    &
    =
    \frac{1}{p-1}-\frac{r}{(p-1)^2}\frac{1}{\log s}
    +O\left(\frac{1}{(\log s)^{2}}\right),
  \end{align}
  as $s\to \infty$. 
  Furthermore, 
  \begin{equation}
    \label{eq:4.10}
    I(\rho)\simeq \frac{1}{\log \phi(\rho)}\simeq \frac{1}{\rho}
  \end{equation}
  for all sufficiently large $\rho>0$ and
  \begin{equation}\label{eq:4.11}
    |\eta(\rho)| + |\eta'(\rho)| 
    =O\left(
      \int_{\rho_0}^{\rho}Q(\rho,\tau)|I(\tau)|\, d\tau
    \right)
    =O\left(\frac{1}{\rho}\right) 
  \end{equation}
  as $\rho\to\infty$. 
\end{example}

\begin{proof}
  \eqref{eq:4.6} is a direct consequence, and \eqref{eq:4.8} and \eqref{eq:4.9} follows from \eqref{eq:4.6} and \eqref{eq:4.7}, 
  so we only show \eqref{eq:4.7}. 
  By integration by parts, we have
  \begin{equation}
    \label{eq:4.12}
    \begin{aligned}
      F(s)
      &
      =\int_s^{\infty}\frac{du}{u^p(\log u)^r}
      =\frac{1}{p-1}\frac{1}{s^{p-1}(\log s)^r}-\frac{r}{p-1}\int_s^{\infty}\frac{du}{u^p(\log u)^{r+1}}
      \\
      &
      =\frac{1}{p-1}\frac{1}{s^{p-1}(\log s)^r}-\frac{r}{(p-1)^2}\frac{1}{s^{p-1}(\log s)^{r+1}}
      +\frac{r(r+1)}{(p-1)^2}\int_s^{\infty}\frac{du}{u^p(\log u)^{r+2}}
    \end{aligned}
  \end{equation}
  for all $s>2$. 
  Since 
  \begin{align*}
    0
    \le 
    \int_s^{\infty}\frac{du}{u^p(\log u)^{r+2}} 
    &
    =
    \frac{1}{p-1}
    \frac{1}{s^{p-1}(\log s)^{r+2}}-\frac{r+2}{p-1}\int_s^{\infty}\frac{du}{u^p(\log u)^{r+3}}
    \lesssim 
    \frac{1}{s^{p-1}(\log s)^{r+2}} 
  \end{align*}
  for all sufficiently large $s>0$, 
  by \eqref{eq:4.12} we obtain \eqref{eq:4.7}. 
  Estimate~\eqref{eq:4.10} can be seen from \eqref{eq:2.4}, \eqref{eq:2.7}, \eqref{eq:2.8}, \eqref{eq:4.7}, \eqref{eq:4.8} and \eqref{eq:4.9}. 
  We can also obtain \eqref{eq:4.11} by simple computations with the help of Theorem~\ref{Theorem:3.1}. 
\end{proof}

\begin{example}
  \label{Example:4.3}
  Let $f(s)=s^p\exp((\log s)^r)$ for $s>1$ with $p>1$ and $0<r<1$. 
  Then $p_f=p$ and 
  \begin{align}
    \notag 
    f'(s)
    &
    =
    ps^{p-1}\exp((\log s)^r)
    +rs^{p-1}(\log s)^{r-1}\exp((\log s)^r),
    \\
    \label{eq:4.13}
    \begin{split}
      F(s)
      &
      =
      \frac{1}{p-1}\frac{1}{s^{p-1}\exp((\log s)^r)}
      -
      \frac{r}{(p-1)^2}
      \frac{
        (\log s)^{r-1}
      }{
        s^{p-1}
        \exp((\log s)^r)
      }
      \\
      & \hspace{14em}
      + O\left(
      \frac{
      (\log s)^{2r-2}
      }
      {
      s^{p-1}\exp((\log s)^r)
      }
      \right),
    \end{split}
    \\
    \notag 
    f'(s)F(s)
    &
    =
    \frac{p}{p-1}
    -\frac{r}{(p-1)^2}(\log s)^{r-1}
    +
    O\left((\log s)^{2r-2}\right),
    \\
    \notag 
    \frac{f(s)F(s)}{s}
    &
    =
    \frac{1}{p-1}
    -\frac{r}{(p-1)^2}(\log s)^{r-1}
    +
    O\left((\log s)^{2r-2}\right),
  \end{align}
  as $s\to \infty$.
  Furthermore, 
  \begin{equation}
    \label{eq:4.14}
    I(\rho)\simeq \frac{1}{(\log \phi(\rho))^{1-r}}\simeq \frac{1}{\rho^{1-r}}
  \end{equation}
  for all sufficiently large $\rho>0$ 
  and
  \begin{equation}\label{eq:4.15}
    |\eta(\rho)| + |\eta'(\rho)| 
    = O\left(\int_{\rho_0}^{\rho}Q(\rho,\tau)|I(\tau)|\, d\tau\right)=O\left(\frac{1}{\rho^{1-r}}\right) 
  \end{equation}
  as $\rho\to\infty$. 
\end{example}

\begin{proof}
  Similarly to examples~\ref{Example:4.1} and \ref{Example:4.2}, 
  we only derive \eqref{eq:4.13}. 
  It follows from integration by parts that
  \begin{align*}
    F(s)
    &
    =
    \int_{s}^{\infty}\frac{du}{u^p\exp((\log u)^r)}
    =\frac{1}{p-1}\frac{1}{s^{p-1}\exp((\log s)^r)}
    -\frac{r}{p-1}\int_s^{\infty}\frac{(\log u)^{r-1}}{u^p\exp((\log u)^r)}\, du
    \\
    &
    =
    \frac{1}{p-1}
    \frac{1}{s^{p-1}\exp((\log s)^r)}
    -\frac{r}{(p-1)^2}
    \frac{(\log s)^{r-1}}{s^{p-1}\exp((\log s)^r)}
    \\
    &
    \qquad
    +
    \frac{r^2}{(p-1)^2}
    \int_s^{\infty}\frac{(\log u)^{2r-2}}{u^p\exp((\log u)^r)}\, du
    -
    \frac{r(r-1)}{(p-1)^2}
    \int_s^{\infty}\frac{(\log u)^{r-2}}{u^p\exp((\log u)^r)}\, du.
  \end{align*}
  Since $r<1$, the functions $(\log u)^{2r-2}$ and $(\log u)^{r-2}$ are non-increasing, 
  so the last two terms can be estimated as follows: 
  \begin{align*}
    \left|
      \frac{r^2}{(p-1)^2}
      \int_s^{\infty}\frac{(\log u)^{2r-2}}{u^p\exp((\log u)^r)}du
      -
      \frac{r(r-1)}{(p-1)^2}
      \int_s^{\infty}\frac{(\log u)^{r-2}}{u^p\exp((\log u)^r)}du
    \right|
    \\
    \qquad 
    \lesssim 
    \frac{(\log s)^{2r-2}}{\exp((\log s)^r)}
    \int_s^{\infty}\frac{1}{u^p}du
    =
    O
    \left(
    \frac{(\log s)^{2r-2}}{s^{p-1}\exp((\log s)^r)}
    \right)
  \end{align*}
  for all sufficiently large $s>0$. 
  This proves \eqref{eq:4.13}. 
  Estimates~\eqref{eq:4.14} and \eqref{eq:4.15} are easily obtained, 
  so the proof is omitted. 
\end{proof}

\begin{example}
  \label{Example:4.4}
  Let $f(s)=s^p+s^r(\log s)^{\beta}$ for $s>2$ with $p>1$, $0<r<p$ and $\beta \in \mathbb{R}$. 
  Then $p_f=p$ and 
  \begin{align}
    \notag 
    f'(s)
    &
    =
    ps^{p-1}
    +rs^{r-1}(\log s)^{\beta}+\alpha s^{r-1}(\log s)^{\beta-1}, 
    \\
    \notag 
    F(s)
    &
    =
    \frac{1}{p-1}\frac{1}{s^{p-1}}
    -
    \frac{1}{(2p-r-1)^2}
    \frac{
    (\log s)^{\beta}
    }
    {
    s^{2p-r-1}
    }
    +O\left(\frac{(\log s)^{\beta-1}}{s^{2p-r-1}}\right),
    \\
    \notag 
    f'(s)F(s)
    &
    =
    \frac{p}{p-1}
    -
    \frac{(p-r)(p-r-1)}{(p-1)(2p-r-1)}
    \frac{(\log s)^{\beta}}{s^{p-r}}
    +
    O\left(\frac{(\log s)^{\beta-1}}{s^{p-r}}
    \right),
    \\
    \notag 
    \frac{f(s)F(s)}{s}
    &
    =
    \frac{1}{p-1}
    +
    \frac{p-r}{(p-1)(2p-r-1)}
    \frac{(\log s)^{\beta}}{s^{p-r}}
    +
    O\left(\frac{(\log s)^{\beta-1}}{s^{p-r}}
    \right),
  \end{align}
  as $s\to \infty$.
  Furthermore, 
  \[
    I(\rho)\simeq \frac{(\log \phi(\rho))^{\beta}}{\phi(\rho)^{p-r}}\simeq \rho^{\beta}e^{-\frac{2(p-r)}{p-1}\rho}
  \]
  for all sufficiently large $\rho>0$ 
  and
  upper decay estimates of $|\eta(\rho)|+|\eta'(\rho)|$ as $\rho\to\infty$
  are given by 
  a constant multiple of the functions in Table~$\mathrm{\ref{table:3}}$.  
  Here $r^*$ is the exponent defined by \eqref{eq:1.16}. 
  \newline
  \begin{table}[H]
    \begin{center}
      {\renewcommand\arraystretch{2}
      \begin{tabular}{|c||c|c|c|}
        \hline 
        $|\eta(\rho)|+|\eta'(\rho)|$
        & $p_\mathrm{c}<p<p_*$ 
        & $p=p_*$ 
        & $p_*<p<p_\mathrm{S}$  
        \\[1mm]
        \hline\hline 
        $\displaystyle 0< r<r^*$ 
        & $\displaystyle e^{-\lambda_1 \rho}$
        & $\displaystyle \rho e^{-\lambda_*\rho}$
        &
        $\displaystyle e^{-\frac{a}{2}\rho}$
        \\[1mm]
        \hline
        \multirow{2}{*}{$\displaystyle r=r^*$} 
        & \multirow{2}{*}{$\displaystyle \rho^{\beta+1} e^{-\lambda_1 \rho}$}
        & $\rho(\log \rho)e^{-\lambda_* \rho}$ if $\beta= -1$
        &
        \multirow{2}{*}{$\displaystyle \rho^{\beta+1} e^{-\frac{a}{2}\rho}$}
        \\ \cdashline{3-3}
        & & $\rho^{\beta+2}e^{-\lambda_* \rho}$ if $\beta\neq -1$ & 
        \\[1mm]
        \hline
        \rule[0cm]{-1.5mm}{8mm}
        \raisebox{3.5pt}{$\displaystyle r^*<r<p$}
        & 
        \multicolumn{3}{c|}{
          \raisebox{3.5pt}{$\displaystyle  \rho^{\beta} e^{-\frac{2(p-r)}{{p-1}}\rho}$}
        }
        \\
        \hline 
      \end{tabular}
      \caption{Upper estimates of decay rates of $|\eta(\rho)|+|\eta'(\rho)|$}\label{table:3}
      }
    \end{center}
  \end{table}
\end{example}

\begin{proof}
  Example~\ref{Example:4.4} can be obtained in the same way as Example~\ref{Example:4.1}.
  We omit the proof. 
\end{proof}

\section*{Appendix}
As mentioned in Remark~\ref{Remark:1.1}, 
in order to find the specific divergence rates of the singular solutions, 
it is necessary to examine the expansion of $F^{-1}(s)$ as $s\to 0$, 
but this needs to be checked separately depending on the form of $f$
and has therefore not been included in the main body of the paper. 
In this section, we focus on the case $f(s)=s^p+s^r$ with $1<r<p$ which appeared in Example~\ref{Example:4.1}, 
and give the behavior of $F^{-1}(\sigma)$ as $\sigma \to 0$ and hence of $u(x)$ as $x\to 0$. 

Set $\sigma =F(s)$.
Note that $\sigma\to 0$ as $s\to \infty$.  
By \eqref{eq:4.2} we have 
\[
  \sigma
  =
  F(s)
  =
  \frac{1}{p-1}\frac{1}{s^{p-1}}
  \left(1-\frac{p-1}{2p-r-1}\frac{1}{s^{p-r}}+O\left(s^{-2(p-r)}\right)\right), 
\]
and as a result, we obtain    
\[
  s^{p-1}
  =
  ((p-1)\sigma)^{-1}
  \left(1-\frac{p-1}{2p-r-1}\frac{1}{s^{p-r}}+O\left(s^{-2(p-r)}\right)\right) 
\]
for all sufficiently large $s>0$. 
Hence we have
\begin{align*}
  s^{p-r}
  &
  =
  ((p-1)\sigma)^{-\frac{p-r}{p-1}}
  \left(1-\frac{p-1}{2p-r-1}\frac{1}{s^{p-r}}+O\left(s^{-2(p-r)}\right)\right)^{\frac{p-r}{p-1}}
  \\
  &
  =
  ((p-1)\sigma)^{-\frac{p-r}{p-1}}
  \left(1-\frac{p-r}{2p-r-1}\frac{1}{s^{p-r}}+O\left(s^{-2(p-r)}\right)\right) 
\end{align*}
for all sufficiently large $s>0$. 
Substituting the expression on the right hand side into the denominator of $1/s^{p-r}$ on the right hand side, we have
\begin{align*}
  s^{p-r}
  &
  =
  ((p-1)\sigma)^{-\frac{p-r}{p-1}}
  \left(1-\frac{p-r}{2p-r-1}((p-1)\sigma)^{\frac{p-r}{p-1}}\frac{1}{1+O\left(\sigma^{\frac{p-r}{p-1}}\right)}
  +O\left(\sigma^{\frac{2(p-r)}{p-1}}\right)\right)
  \\
  &
  =
  ((p-1)\sigma)^{-\frac{p-r}{p-1}}-\frac{p-r}{2p-r-1}
  +O\left(\sigma^{\frac{p-r}{p-1}}\right)
\end{align*}
for all sufficiently small $\sigma>0$. 
Therefore,
\begin{align*}
  F^{-1}(\sigma)
  =
  s
  &
  =
  \left(
    ((p-1)\sigma)^{-\frac{p-r}{p-1}}-\frac{p-r}{2p-r-1}
    +O\left(\sigma^{\frac{p-r}{p-1}}\right)
  \right)^{\frac{1}{p-r}}
  \\
  &
  =
  ((p-1)\sigma)^{-\frac{1}{p-1}}-\frac{1}{2p-r-1}((p-1)\sigma)^{\frac{p-r-1}{p-1}}
  +O\left(\sigma^{\frac{2(p-r)-1}{p-1}}\right)
\end{align*}
for all sufficiently small $\sigma>0$. 
Taking $\sigma = |x|^2/(2N-4q)$ with $q={p}/{(p-1)}$, we obtain
\[
  F^{-1}\left(\frac{|x|^2}{2N-4q}\right)
  =
  L_p|x|^{-\frac{2}{p-1}}
  -\frac{L_p^{-(p-r-1)}}{2p-r-1}
  |x|^{-\frac{2}{p-1}+\frac{2(p-r)}{p-1}}+O\left(|x|^{-\frac{2}{p-1}+\frac{4(p-r)}{p-1}}\right),
\]
where $L_p$ is the constant given in \eqref{eq:1.4}.  
Thus, by Corollary~\ref{Corollary:1.1} we see that there exists a singular solution $u$ of \eqref{eq:1.1} satisfying 
\begin{align*}
  u(x)
  & 
  =
  \left(
    L_p |x|^{-\frac{2}{p-1}}
    -\frac{L_p^{-(p-r-1)}}{2p-r-1}
    |x|^{-\frac{2}{p-1}+\frac{2(p-r)}{p-1}}+O\left(|x|^{-\frac{2}{p-1}+\frac{4(p-r)}{p-1}}\right)
  \right) (1+\theta(|x|))
  \\
  &
  =
  L_p|x|^{-\frac{2}{p-1}}
  \left(
    1+\theta(|x|)
    -\frac{L_p^{-(p-r-1)}}{2p-r-1}
    |x|^{\frac{2(p-r)}{p-1}}+O\left(|x|^{\frac{2(p-r)}{p-1}}\theta(|x|)\right)
    +O\left(|x|^{\frac{4(p-r)}{p-1}}\right)
  \right)
\end{align*}
as $x\to 0$, 
where $\theta$ is estimated from above by a constant multiple of the functions in Table~\ref{table:1}. 

In other examples, the behavior of $F^{-1}(\sigma)$ as $\sigma \to 0$ can 
be obtained from the expansion of $F(s)$ as $s\to \infty$ given in Examples~\ref{Example:4.2}--\ref{Example:4.4}, 
but this will be left to the reader. 

In order to obtain a higher order expansion of singular solution, 
we need to study the behavior of $\theta(|x|)$ as $x\to 0$, 
that is, a solution of \eqref{eq:2.11} $\eta(\rho)$ 
as $\rho\to\infty$ in more detail, 
which will be investigated in a forthcoming paper. 

\bigskip


\noindent
\textbf{Acknowledgements}\quad 
The authors were partially supported by JSPS KAKENHI Grant Number 23K03179 and 20KK0057,
respectively.



\bigskip

\noindent 
\textbf{Data availability}\quad 
Data sharing is not applicable to this article as no new data were created or analyzed in this study.

\bigskip

\noindent 
\textbf{Conflict of interest}\quad 
The authors declare no conflict of interest.




\end{document}